\documentclass[a4paper,reqno,11pt]{amsart}
\usepackage{amssymb,amsfonts,verbatim,mathtools}
\usepackage[alwaysadjust]{paralist}
\usepackage{hyperref}

\hypersetup{
    colorlinks,
    linkcolor={blue},
    citecolor={blue},
    urlcolor={blue}
}
\usepackage[all]{xy}
\usepackage{pstricks}
\usepackage{auto-pst-pdf}
\usepackage{pstricks-add}
\usepackage{mathptmx}
\renewcommand{\P}{\mathcal{P}}
\newcommand{\C}{{\mathbb C}}
\newcommand{\R}{{\mathbb R}}
\newcommand{\Z}{{\mathbb Z}}

\newcommand{\n}{\nabla}
\newcommand{\im}{\operatorname{im}}
\renewcommand{\phi}{\varphi}
\newcommand{\eps}{\varepsilon}
\newcommand{\spec}{\operatorname{spec}}
\newcommand{\Cliff}{\operatorname{Cliff}}

\newcommand{\Hom}{\operatorname{Hom}}

\newcommand{\dom}{\operatorname{dom}}
\newcommand{\Bad}{B^\mathrm{ad}}
\newcommand{\BAPS}{B_\mathrm{APS}}
\newcommand{\BmAPS}{B_\mathrm{mAPS}}
\newcommand{\ad}{\mathrm{ad}}

\newcommand{\ind}{\operatorname{ind}}
\newcommand{\<}{\left\langle}       
\renewcommand{\>}{\right\rangle}       

\newcommand{\id}{{\operatorname{id}}}

\newcommand{\dM}{{\partial M}}
\newcommand{\intM}{{\mathring{M}}}
\newcommand{\LzM}[1]{\| #1 \|_{L^2(M)}}

\newcommand{\Cucc}{C^\infty_{cc}}
\newcommand{\Cu}{C^\infty}

\newcommand{\nor}{\mathrm{nor}}
\newcommand{\abs}{\mathrm{abs}}
\newcommand{\rel}{\mathrm{rel}}
\newcommand{\nd}{\nabla^\partial}
\newcommand{\loc}{\operatorname{loc}}
\newcommand{\dV}{\operatorname{dV}}
\newcommand{\dS}{\operatorname{dS}}

\newtheorem{thm}{Theorem}[section]
\newtheorem{lemma}[thm]{Lemma}
\newtheorem{prop}[thm]{Proposition}
\newtheorem{cor}[thm]{Corollary}

\theoremstyle{definition}
\newtheorem{remark}[thm]{Remark}
\newtheorem{remarks}[thm]{Remarks}
\newtheorem{definition}[thm]{Definition}
\newtheorem{definitions}[thm]{Definitions}

\newtheorem{example}[thm]{Example}
\newtheorem{examples}[thm]{Examples}

\newtheorem{stase}[thm]{Standard Setup}

\newcommand{\dref}[1]{Definition~\ref{#1}}
\newcommand{\cref}[1]{Corollary~\ref{#1}}
\newcommand{\lref}[1]{Lemma~\ref{#1}}
\newcommand{\pref}[1]{Proposition~\ref{#1}}
\newcommand{\tref}[1]{Theorem~\ref{#1}}

\newcommand{\eref}[1]{Example~\ref{#1}}

\newcommand{\setup}{Standard Setup~\ref{setup}}

\newtagform{simple}{}{}{}

\long\def\symbolfootnote[#1]#2{\begingroup%
\def\thefootnote{\fnsymbol{footnote}}\footnote[#1]{#2}\endgroup} 

\begin{document}


\title
[Elliptic Boundary Value Problems]
{Guide to Elliptic Boundary Value Problems \\ for Dirac-Type Operators}

\author{Christian B{\"a}r}
\author{Werner Ballmann}

\address{
Institut f\"ur Mathematik\\
Universit{\"a}t Potsdam\\
Karl-Liebknecht-Str. 24--25\\
Haus 9\\
14476 Potsdam\\
Germany
}
\address{
Max Planck Institute for Mathematics\\
Vivatsgasse 7\\
53111 Bonn\\
Germany
}
 
\email{\href{mailto:baer@math.uni-potsdam.de}{baer@math.uni-potsdam.de}}
\email{\href{mailto:hwbllmnn@mpim-bonn.mpg.de}{hwbllmnn@mpim-bonn.mpg.de}}

\dedicatory{Dedicated to the memory of Friedrich Hirzebruch}



\thanks{}


\begin{abstract}
We present an introduction to boundary value problems for Dirac-type operators on complete Riemannian manifolds with compact boundary.
We introduce a very general class of boundary conditions which contains local elliptic boundary conditions in the sense of Lopatinski and Shapiro as well as the Atiyah-Patodi-Singer boundary conditions.
We discuss boundary regularity of solutions and also spectral and index theory.
The emphasis is on providing the reader with a working knowledge.
\end{abstract}

\maketitle

\section{Introduction}
\label{secIn}

Boundary value problems for elliptic differential equations of second order, such as the Dirichlet problem for harmonic functions, have been the object of intense investigation since the 19th century.
For a large class of such problems, the analysis is by now classical and well understood.
There are numerous applications in and outside mathematics.

The situation is much less satisfactory for boundary value problems for first-order elliptic differential operators such as the Dirac operator.
Let us illustrate the phenomena that arise with the elementary example of holomorphic functions on the closed unit disk $D\subset\C$.
Holomorphic functions are the solutions of the elliptic equation $\bar\partial f=0$.
The real and imaginary parts of $f$ are harmonic and they determine each other up to a constant.
Thus for most smooth functions $g:\partial D\to \C$, the Dirichlet problem $\bar\partial f=0$, $f|_{\partial D}=g$, is not solvable.
Hence such a boundary condition is too strong for first-order operators.

Ideally, a ``good'' boundary condition should ensure that the equation $\bar\partial f=h$ has a unique solution for given $h$.
At least we want to have that the kernel and the cokernel of $\bar\partial$ become finite dimensional, more precisely, that $\bar\partial$ becomes a Fredholm operator.
If we expand the boundary values of $f$ in a Fourier series, $f(e^{it})=\sum_{k=-\infty}^\infty a_k e^{ikt}$, then we see $a_{-1}=a_{-2}=\ldots=0$ because otherwise $f$ would have a pole at $z=0$.
Therefore it suffices to impose $a_0=a_1=a_2=\ldots=0$ to make the kernel trivial.
Similarly, imposing $a_k=a_{k+1}=a_{k+2}=\ldots=0$ would make the kernel $k$-dimensional.
These are typical examples for the nonlocal boundary conditions that one has to consider when dealing with elliptic operators of first order.

A major break-through towards a general theory was achieved in the seminal article \cite{APS}, where Atiyah, Patodi and Singer obtain an index theorem for a certain class of first-order elliptic differential operators on compact manifolds with boundary. 
This work lies at the heart of many investigations concerning boundary value problems and $L^2$-index theory for first-order elliptic differential operators.

The aim of the present paper is to provide an introduction to the general theory of elliptic boundary value problems for Dirac-type operators and to give the reader a sound working knowlegde of this material.
To a large extent, we follow \cite{BB} where all details are worked out but, due to its length and technical complexity, that article may not be a good first start.
Results which we only cite here are marked by a $\blacksquare$.
The present paper also contains new additions to the results in \cite{BB}; they are given full proofs, terminated by a $\Box$.
For previous results and alternative approaches see the list of references in \cite{BB}.

After some preliminaries on differential operators in Section~\ref{secIntro},
we discuss Dirac-type operators in Section~\ref{susedity}.
An important class consists of Dirac operators in the sense of Gromov and Lawson  \cite{GL,LM} associated with Dirac bundles.
In Section~\ref{secdity},
we introduce boundary value problems for Dirac-type operators as defined in \cite{BB}.
We discuss their regularity theory.
For instance, \tref{regula} applied to $\bar\partial$ tells us, that, for given $h\in C^\infty(D,\C)$, any solution $f$ of $\bar\partial f=h$ satisfying the boundary conditions described above will be smooth up to the boundary.
We explain that the classical examples,
like local elliptic boundary conditions in the sense of Lopatinski and Shapiro and the boundary conditions introduced by Atiyah, Patodi, and Singer,
belong to our class of boundary value problems.
This class also contains examples which cannot be described by pseudo-differential operators.
In Section~\ref{spec}, we investigate the spectral theory associated with boundary conditions.
The index theory for boundary value problems is the topic of Section~\ref{inth}.
In general, we assume that the underlying manifold $M$ is a complete, not necessarily compact, Riemannian manifold with compact boundary.
We discuss coercivity conditions which ensure the Fredholm property also for noncompact $M$.

\section{Preliminaries}
\label{secIntro}

Let $M$ be a Riemannian manifold with compact boundary $\partial M$
and interior unit normal vector field $\nu$ along $\partial M$.
The Riemannian volume element on $M$ will be denoted by $\dV$, the one on $\dM$ by $\dS$.
Denote the interior part of $M$ by $\intM$.

For a vector bundle $E$ over $M$
denote by $C^\infty(M,E)$ the space of smooth sections of $E$
and by $C^\infty_c(M,E)$ and $C^\infty_{cc}(M,E)$ the subspaces of $C^\infty(M,E)$
which consist of smooth sections with compact support in $M$ and $\intM$,
respectively.
Let $L^2(M,E)$ be the Hilbert space (of equivalence classes)
of square-integrable sections of $E$
and $L^2_{\loc}(M,E)$ be the space of locally square-integrable sections of $E$.
For any integer $k\ge0$,
denote by $H^k_{\loc}(M,E)$ the space of sections of $E$
which have weak derivatives up to order $k$
(with respect to some or any connection on $E$)
that are locally square-integrable.

\subsection{Differential operators}
\label{susediffop}

Let $E$ and $F$ be Hermitian vector bundles over $M$ and
\begin{equation*}
  D: C^\infty(M,E) \to C^\infty(M,F)
\end{equation*}
be a differential operator of order (at most) $\ell$ from $E$ to $F$.
For simplicity, we only consider the case of complex vector bundles.
If $D$ acts on real vector bundles one can complexify and thus reduce to the complex case.

Denote by $D^*$ the {\em formal adjoint} of $D$.
This is the unique differential operator of order (at most) $\ell$ from $F$ to $E$ such that
\begin{equation*}
  \int_M \<D\Phi,\Psi\>\dV = \int_M \<\Phi,D^*\Psi\>\dV ,
\end{equation*}
for all $\Phi\in C^\infty_{cc}(M,E)$ and $\Psi\in C^\infty(M,F)$.
We say that $D$ is {\em formally self-adjoint} if $E=F$ and $D=D^*$.
 
Consider $D$ as an unbounded operator, $D_{cc}$,
from $L^2(M,E)$ to $L^2(M,F)$ with domain $\dom D_{cc}=C^\infty_{cc}(M,E)$,
and similarly for $D^*$.
The {\em minimal extension} $D_{\min}$ of $D$ is obtained by taking the
closure of the graph of $D_{cc}$ in $L^2(M,E)\oplus L^2(M,F)$.
In other words, $\Phi\in L^2(M,E)$ belongs to the domain $\dom D_{\min}$ of $D_{\min}$
if there is a sequence $(\Phi_n)$ in $C^\infty_{cc}(M,E)$ which converges
to $\Phi$ in $L^2(M,E)$ such that $(D\Phi_n)$ is a Cauchy sequence in $L^2(M,F)$;
then we set $D_{\min}\Phi:=\lim_n D\Phi_n$.
By definition, $C^\infty_{cc}(M,E)$ is dense in $\dom D_{\min}$
with respect to the graph norm of $D_{\min}$.
The {\em maximal extension} $D_{\max}$ of $D$
is defined to be the adjoint operator of $D^*_{cc}$,
that is, $\Phi$ in $L^2(M,E)$ belongs to the domain $\dom D_{\max}$ of $D_{\max}$
if there is a section $\Xi\in L^2(M,F)$ such that $D\Phi=\Xi$ in the sense of distributions:
\begin{equation*}
  \int_M \<\Xi,\Psi\>\dV = \int_M\<\Phi,D^*\Psi\> \dV,
\end{equation*}
for all $\Psi\in C^\infty_{cc}(M,F)$; then we set $D_{\max}\Phi:=\Xi$.
In other words, $(\Phi,-\Xi)$ is perpendicular to the graph of $D^*_{cc}$
in $L^2(M,E)\oplus L^2(M,F)$.
Equivalently, $(\Phi,-\Xi)$ is perpendicular to the graph of $D^*_{\min}$
in $L^2(M,E)\oplus L^2(M,F)$.
It is easy to see that
\begin{equation*}
  D_{\min} \subset D_{\max}
\end{equation*}
in the sense that $\dom D_{\min}\subset\dom D_{\max}$
and $D_{\max}|_{\dom D_{\min}}=D_{\min}$.
By definition, $D_{\min}$ and $D_{\max}$ are {\em closed operators},
meaning that their graphs are closed subspaces of $L^2(M,E)\oplus L^2(M,F)$.
Hence the {\em graph norm}, that is, the norm associated with the scalar product
\begin{equation*}
  ( \Phi,\Psi )_D
  := \int_M ( \langle \Phi,\Psi \rangle + \langle D_{\max}\Phi,D_{\max}\Psi \rangle )\dV ,
\end{equation*}
turns $\dom D_{\min}$ and $\dom D_{\max}$ into Hilbert spaces.
Boundary value problems in our sense are concerned with closed operators
lying between $D_{\min}$ and $D_{\max}$.

\subsection{The principal symbol}
For a differential operator $D$ from $E$ to $F$ of order (at most) $\ell$ as above,
there is a field $\sigma_D:(T^*M)^{\ell}\to\Hom(E,F)$
of symmetric $\ell$-linear maps,
the {\em principal symbol} $\sigma_D$ of $D$,
defined by the $\ell$-fold commutator\footnote{Here $[D,f]=D\circ (f\cdot\id_E) - (f\cdot\id_F)\circ D$.}
\begin{equation*}
  \sigma_D(df_1,\dots,df_\ell) := \frac1{\ell!} [\dots[D,f_1],\dots,f_\ell] ,
\end{equation*}
for all $f_1,\dots,f_\ell\in C^\infty(M,\R)$.
In the case $\ell=1$, this means that
\begin{equation*}
  D(f\Phi) = \sigma_D(df)\Phi + fD\Phi ,
\end{equation*}
for all $f\in C^\infty(M,\R)$ and $\Phi\in C^\infty(M,E)$.
The principal symbol  $\sigma_D$ vanishes precisely at those points where the order of $D$ is at most $\ell-1$.
The principal symbol of $D^*$ is
\begin{equation}
  \sigma_{D^*}(\xi_1,\dots,\xi_k) = (-1)^\ell \sigma_D(\xi_1,\dots,\xi_\ell)^* ,
  \label{symbad}
\end{equation}
for all $\xi_1,\dots,\xi_\ell\in T^*M$.
Since $\sigma_D$ is symmetric in $\xi_1,\ldots,\xi_\ell$,
it is determined by its values along the diagonal;
we use $\sigma_D(\xi)$ as a shorthand notation for $\sigma_D(\xi,\dots,\xi)$.
Then we have, for all $\xi\in T^*M$,
\begin{equation}
  \sigma_{D_1D_2}(\xi) = \sigma_{D_1}(\xi)\circ\sigma_{D_2}(\xi) 
 \label{symcom}
\end{equation}
for the principal symbol of the composition of differential operators
$D_1$ of order $\ell_1$ and $D_2$ of order $\ell_2$.

The Riemannian metric induces a vector bundle isomorphism $TM \to T^*M$, $X \mapsto X^\flat$, defined by $\langle X,Y \rangle = X^\flat(Y)$ for all $Y$.
The inverse isomorphism $T^*M \to TM$ is denoted by $\xi\mapsto \xi^\sharp$.

\usetagform{simple}

\begin{prop}[Green's formula]\label{greenfor}
Let $D$ be a differential operator from $E$ to $F$ of order one.
Then we have, for all $\Phi\in C^\infty_c(M,E)$ and $\Psi\in C^\infty_c(M,F)$,
\begin{align*}
\int_M \<D\Phi,\Psi\>\dV
&= 
\int_M \<\Phi,D^*\Psi\>\dV - \int_{\partial M} \langle\sigma_D(\nu^\flat)\Phi,\Psi\rangle\dS .
\tag{$\blacksquare$}
\end{align*}
\end{prop}

\usetagform{default}

For a proof see e.g.\ \cite[Prop.~9.1, p.~160]{Ta}.

\begin{examples}\label{exadiffop}
By definition, a connection $\nabla$ on $E$ is a differential operator
from $E$ to $T^*M\otimes E$ of order one such that $[\nabla,f](\Phi)=df\otimes\Phi$.
We obtain
\begin{equation}
  \sigma_{\nabla}(\xi)(\Phi) = \xi\otimes\Phi
  \quad\text{and}\quad
  \sigma_{\nabla^*}(\xi)(\Psi) = - \Psi(\xi^\sharp) .
  \label{symnn}
\end{equation}
Hence all connections on $E$ have the same principal symbol reflecting the fact that the difference of two connections is of order zero.

There are two natural differential operators of order two associated with $\nabla$,
the second covariant derivative $\nabla^2$
with principal symbol
\begin{equation}
  \sigma_{\nabla^2}(\xi)(\Phi) = \xi\otimes\xi\otimes\Phi
  \label{symn2}
\end{equation}
and the connection Laplacian $\nabla^*\nabla$ with principal symbol
\begin{equation}
  \sigma_{\nabla^*\nabla}(\xi)(\Phi)
  = - |\xi|^2 \Phi ,
  \label{symcl}
\end{equation}
and both, \eqref{symn2} and \eqref{symcl}, 
are in agreement with \eqref{symcom} and \eqref{symnn}.
\end{examples}

\subsection{Elliptic operators}
We say that $D$ is {\em elliptic} if $\sigma_D(\xi):E_x\to F_x$
is an isomorphism, for all $x\in M$ and nonzero $\xi\in T^*_xM$.
In the above examples, $\nabla$, $\nabla^*$, and $\nabla^2$ are not elliptic;
in fact, the involved bundles have different rank.
On the other hand, the connection Laplacian is elliptic, by \eqref{symcl}.

Suppose that $D$ is elliptic.
Then {\em interior elliptic regularity} says that, for any given integer $k\ge0$,
$\Phi\in\dom D_{\max}$ is contained in $H^{k+\ell}_{\loc}(\intM,E)$
if $D_{\max}\Phi$ belongs to $H^k_{\loc}(\intM,F)$.
In particular, if $\Phi\in\dom D_{\max}$ satisfies $D_{\max}\Phi\in C^\infty(\intM,F)$,
then $\Phi\in C^\infty(\intM,E)$.

If $M$ is closed and $D$ is elliptic and formally self-adjoint,
then the eigenspaces of $D$ are finite dimensional, contained in $C^\infty(M,E)$,
pairwise perpendicular with respect to the $L^2$-product, and span $L^2(M,E)$.
As an example, the connection Laplacian is elliptic and formally self-adjoint.

For any differential operator $D:C^\infty(M,E) \to C^\infty(M,F)$ of order one, consider the fiberwise linear bundle map
\[
\mathcal{A}_D: T^*M \otimes \Hom(E,E) \to \Hom(E,F),\quad 
V\mapsto \sum\nolimits_j \sigma_D(e_j^*)\circ V(e_j).
\]
Here $(e_1,\dots,e_n)$ is any local tangent frame and  $(e_1^*,\dots,e_n^*)$ its associated dual cotangent frame of $M$.
Note that $\mathcal{A}_D$ does not depend on the choice of frame.

\begin{prop}\label{prop:guterZshg}
Let $D:C^\infty(M,E) \to C^\infty(M,F)$ be a differential operator of order one such that $\mathcal{A}_D$ is onto.
Then there exists a connection $\nabla$ on $E$ such that 
\begin{equation*}
  D = \sum\nolimits_j \sigma_D(e_j^*)\circ\nabla_{e_j} ,
\end{equation*}
for any local tangent frame $(e_1,\dots,e_n)$
and the associated dual cotangent frame $(e_1^*,\dots,e_n^*)$ of $M$.
\end{prop}

The proof can be found in Appendix~\ref{sec:beweise}.

If $D$ is elliptic, $\mathcal{A}_D$ is onto: given $U\in \Hom(E,F)$ put $V(e_2)=\ldots=V(e_n)=0$ and $V(e_1)=\sigma_D(e_1^*)^{-1}\circ U$, for instance.
Hence \pref{prop:guterZshg} applies and we have

\begin{cor}\label{cor:guterZshg}
Let $D:C^\infty(M,E) \to C^\infty(M,F)$ be an elliptic differential operator of order one.
Then there exists a connection $\nabla$ on $E$ such that 
\begin{equation*}
  D = \sum\nolimits_j \sigma_D(e_j^*)\circ\nabla_{e_j} ,
\end{equation*}
for any local tangent frame $(e_1,\dots,e_n)$
and the associated dual cotangent frame $(e_1^*,\dots,e_n^*)$ of $M$.\qed
\end{cor}

In the special case of Dirac-type operators (see definition below), this corollary is \cite[Lemma~2.1]{AT}.
\pref{prop:guterZshg} is also useful for nonelliptic operators.
For instance, it applies to Dirac-type operators on Lorentzian manifolds; these are hyperbolic instead of elliptic.

\section{Dirac-type operators}
\label{susedity}

From now on we concentrate on an important special class of first-order elliptic operators.

\subsection{Clifford relations and Dirac-type operators}

We say that a differential operator $D: C^\infty(M,E) \to C^\infty(M,F)$ of order one
is of {\em Dirac type} if its principal symbol $\sigma_D$
satisfies the {\em Clifford relations},
\begin{align}
  \sigma_D(\xi)^*\sigma_D(\eta) + \sigma_D(\eta)^*\sigma_D(\xi)
  &= 2\<\xi,\eta\> \cdot \id_{E_x} ,
  \label{Cliff1} \\
  \sigma_D(\xi)\sigma_D(\eta)^* + \sigma_D(\eta)\sigma_D(\xi)^*
  &= 2\<\xi,\eta\> \cdot \id_{F_x} ,
  \label{Cliff2}
\end{align}
for all $x\in M$ and $\xi,\eta \in T^*_xM$.

The classical Dirac operator on a spin manifold is an important example.
More generally, the class of Dirac-type operators contains Dirac operators
on Dirac bundles as in  \cite[Ch.~II, \S~5]{LM}.

By \eqref{symbad}, if $D$ is of Dirac type, then so is $D^*$.
Furthermore, by \eqref{Cliff1} and \eqref{Cliff2},
Dirac-type operators are elliptic with
\begin{equation}
  \sigma_D(\xi)^{-1} = |\xi|^{-2}\sigma_D(\xi)^{*} ,
  \text{  for all nonzero $\xi\in T^*M$.}
  \label{symbdt}
\end{equation}
If $D$ is a formally self-adjoint operator of Dirac type on $E$,
then the endomorphisms $\sigma_{D}(\xi)$ are skewhermitian, $\xi\in T^*M$.
In this case, the Clifford relations \eqref{Cliff1} and \eqref{Cliff2}
may be spelled out as
\begin{equation*}
  \sigma_D(\xi)\sigma_D(\eta) + \sigma_D(\eta)\sigma_D(\xi)
= -2\<\xi,\eta\> \cdot \id_{E_x} ,
\end{equation*}
for all $x\in M$ and $\xi,\eta \in T^*_xM$.
In other words, the principal symbol turns $E$ into a bundle of modules
over the Clifford algebras $\Cliff(T^*M)$.

\begin{prop}[Weitzenb\"ock formula]\label{Weitzen}
Let $D:C^\infty(M,E) \to C^\infty(M,F)$ be of Dirac type.
Then there exists a unique metric connection $\nabla$ on $E$ with
\begin{equation}
D^*D = \nabla^*\nabla + \mathcal{K} ,
\label{eq:Weitzenboeck}
\end{equation}
where $\mathcal{K}$ is a field of symmetric endomorphisms of $E$.
\end{prop}

See Appendix~\ref{sec:beweise} for the proof.
For special choices for $D$, this formula is also known as Bochner formula, Bochner-Kodaira formula or Lichnerowicz formula.

In general, the connections in \cref{cor:guterZshg} and \pref{Weitzen} do not coincide.

\subsection{Adapted operators on the boundary}
\label{subsec:adapted}

Suppose from now on that $D$ is of Dirac type.
For $x\in\partial M$, identify $T^*_x\partial M$
with the space of covectors $\xi$ in $T^*_xM$ such that $\xi(\nu(x))=0$.
Then, by \eqref{Cliff1} and \eqref{symbdt},
\begin{equation}
  \sigma_D(\nu(x)^\flat)^{-1} \circ \sigma_D(\xi): E_x \to E_x 
\label{eq:symbDT}
\end{equation}
is skewhermitian, for all $x\in\partial M$ and $\xi\in T^*_x\partial M$.
Hence there exist formally self-adjoint differential operators
$A:C^\infty(\dM,E)\to C^\infty(\dM,E) $ of first order with principal symbol 
\begin{equation}
  \sigma_A(\xi) = \sigma_D(\nu(x)^\flat)^{-1} \circ \sigma_D(\xi) .
  \label{symba}
\end{equation}
We call such operators {\em adapted} to $D$.
Note that such an operator $A$ is also of Dirac type
and that the zero-order term of $A$ is only unique up to addition of a field
of hermitian endomorphisms of $E$.
By \eqref{symbad} and \eqref{eq:symbDT} applied to $D^*$,
the principal symbol of an operator $\tilde A$ adapted to $D^*$ is
\begin{equation*}
\sigma_{\tilde A}({\xi}) 
=
(-\sigma_D(\nu(x)^\flat)^{-1})^* \circ (-\sigma_D(\xi))^{*}
=
\sigma_D(\nu(x)^\flat) \circ \sigma_D(\xi)^{*} .
\end{equation*}
By \eqref{symba}, this implies
\begin{align*}
\sigma_{\tilde A}({\xi}) 
&=
\sigma_D(\nu(x)^\flat) \circ (\sigma_D(\nu(x)^\flat)\circ\sigma_A(\xi) )^* \\
&=
\sigma_D(\nu(x)^\flat) \circ \sigma_A(\xi)^*\circ\sigma_D(\nu(x)^\flat)^*\\
&=
\sigma_D(\nu(x)^\flat) \circ \sigma_{-A}(\xi)\circ\sigma_D(\nu(x)^\flat)^{-1}.
\end{align*}
Hence, if $A$ is adapted to $D$, then
\begin{equation}
  \tilde A = \sigma_D(\nu^\flat) \circ (-A) \circ \sigma_D(\nu^\flat)^{-1}
  \label{atilde}
\end{equation}
is adapted to $D^*$.
Given $A$, this choice of $\tilde A$ is the most natural one.

\subsection{Formally self-adjoint Dirac-type operators}

If the Dirac-type operator $D$ is formally self-adjoint, then there is a particularly useful choice of adapted boundary operator $A$.

\begin{lemma}\label{norma}
Let $D:C^\infty(M,E)\to C^\infty(M,E)$ be a formally self-adjoint operator of Dirac type.
Then there is an operator $A$ adapted to $D$ along $\partial M$
such that $\sigma_D(\nu^\flat)$ anticommutes with $A$,
\begin{equation}
  \sigma_D(\nu^\flat)\circ A = - A\circ \sigma_D(\nu^\flat) .
\label{eq:ANuAnti}
\end{equation}
\end{lemma}

See Appendix~\ref{sec:beweise} for the proof.

\begin{remarks}\label{remnorma}
\begin{inparaenum}[1)]
\item 
The operator $A$ in \lref{norma} is unique up to addition of a field
of symmetric endomorphisms of $E$ along $\partial M$
which anticommutes with $\sigma_D(\nu^\flat)$. \\
\item
If $A$ anticommutes with $\sigma_D(\nu^\flat)$,
then $\sigma_D(\nu^\flat)$ induces isomorphisms
between the $\pm\lambda$-eigenspaces of $A$, for all $\lambda\in\R$.
In particular, $\ker A$ is invariant under $\sigma_D(\nu^\flat)$ and the $\eta$-invariant\footnote{The $\eta$-invariant of $A$ is defined as the value of the meromorphic extension of $\eta(s)=\sum_{\lambda\neq0}\mathrm{sign}(\lambda)|\lambda|^{-s}$ at $s=0$, see \cite{APS}. Here the sum is taken over all nonzero eigenvalues of $A$ taking multiplicities into account. Hence the $\eta$-invariant is a measure for the asymmetry of the spectrum.} of $A$ vanishes.
Moreover,
\begin{equation*}
  \omega(\varphi,\psi) := (\sigma_D(\nu^\flat)\varphi,\psi)_{L^2(\dM)}
\end{equation*} 
is a nondegenerate skewhermitian form on $\ker A$ (and also on $L^2(\partial M,E)$).
\end{inparaenum}
\end{remarks}

\section{Boundary value problems}
\label{secdity}

In this section we will study boundary value problems.
This will be done under the following

\begin{stase}\label{setup}
\hspace{10mm}
\vspace{-3mm}
\begin{itemize}[$\diamond$]
\item 
$M$ is a complete Riemannian manifold with compact boundary $\dM$;
\item
$\nu$ is the interior unit normal vector field along $\dM$;
\item
$E$ and $F$ are Hermitian vector bundles over $M$;
\item
$D:C^\infty(M,E)\to C^\infty(M,F)$ is a Dirac-type operator;
\item
$A:C^\infty(\dM,E)\to C^\infty(\dM,E)$ is a boundary operator adapted to $D$.
\end{itemize}
\end{stase}

\subsection{Spectral subspaces}

If $A$ is adapted to $D$, then $A$ is a formally self-adjoint elliptic operator over the compact manifold $\dM$.
Hence we have, in the sense of Hilbert spaces,
\begin{equation*}
  L^2(\partial M,E) = \oplus_j\, \C\cdot\varphi_j ,
\end{equation*}
where $(\varphi_j)$ is an orthonormal basis of $L^2(\partial M,E)$ consisting of eigensections of $A$, $A\varphi_j=\lambda_j\varphi_j$.
In terms of such an orthonormal basis, 
the Sobolev space $H^s(\partial M,E)$, $s\in\R$, 
consists of all sections
\begin{equation*}
  \varphi = \sum\nolimits_j a_j\varphi_j
  \quad\text{such that}\quad
  \sum\nolimits_j |a_j|^2(1+\lambda_j^{2})^s < \infty ,
\end{equation*}
where $L^2(\partial M,E)=H^0(\partial M,E)$.
The natural pairing
\begin{equation}
  H^s(\partial M,E) \times H^{-s}(\partial M,E) \to \C, \quad
  \big(\sum\nolimits_j a_j\phi_j,\sum\nolimits_j b_j\phi_j\big)
  = \sum\nolimits_j \bar a_jb_j ,
  \label{perfect}
\end{equation}
is perfect, for all $s\in\R$.
By the Sobolev embedding theorem,
\begin{equation*}
  C^\infty(\partial M,E) =  \bigcap_{s\in\R}H^s(\partial M,E) .
\end{equation*}
Rellich's embedding theorem says that for $s_1>s_2$ the embedding
\begin{equation*}
 H^{s_1}(\dM,E) \hookrightarrow H^{s_2}(\dM,E)
\end{equation*}
is compact.
We also set
\begin{equation*}
  H^{-\infty}(\partial M,E) := \bigcup_{s\in\R}H^s(\partial M,E) .
\end{equation*}
For $I\subset\R$, let $Q_I$ be the associated spectral projection,
\begin{equation}
  Q_I: \sum\nolimits_j a_j\varphi_j \mapsto
  \sum\nolimits_{\lambda_j\in I} a_j\varphi_j .
  \label{specpro}
\end{equation}
Then $Q_I$ is orthogonal and maps $H^s(\partial M,E)$ to itself,
for all $s\in\R$.
Set
\begin{equation*}
  H^s_I(A) :=Q_I(H^s(\partial M,E)) \subset H^s(\partial M,E) .
\end{equation*}
For $a\in\R$, define the hybrid Sobolev spaces
\begin{align}
  \check H(A) &:=H^{1/2}_{(-\infty,a)}(A) \oplus H^{-1/2}_{[a,\infty)}(A) , 
  \label{checka}\\
  \hat H(A) &:=H^{-1/2}_{(-\infty,a)}(A) \oplus H^{1/2}_{[a,\infty)}(A) .
  \label{hata}
\end{align}
Note that, as topological vector spaces,
$\check H(A)$ and $\hat H(A)$ do not depend on the choice of $a$.
In particular,
\begin{equation*}
  \hat H(A) = \check H(-A) .
\end{equation*}
Moreover, the natural pairing
\begin{equation*}
  \check H(A) \times \check H(-A)\to\C , \quad
  \big(\sum\nolimits_j a_j\phi_j,\sum\nolimits_j b_j\phi_j\big)
  = \sum\nolimits_j \bar a_jb_j ,
\end{equation*}
is perfect, compare \eqref{perfect}.

\subsection{The maximal domain}

Following \cite[Cor.~6.6, Thm.~6.7, Prop.~7.2]{BB}, we now discuss properties of the maximal domain of $D$.

\usetagform{simple}

\begin{thm}\label{maxdom}
Assume the \setup.
Then the domain of $D_{\max}$, equipped with the graph norm topology,
has the following properties:
\begin{asparaenum}[1)]
\item\label{ccr}
$C^\infty_c(M,E)$ is dense in $\dom D_{\max}$;
\item\label{trace} 
the trace map $\mathcal R\Phi:=\Phi|_{\partial M}$ on $C^\infty_c(M,E)$
extends uniquely to a continuous surjection $\mathcal R: \dom D_{\max} \to \check H(A)$;
\item\label{maxmin}
$\dom D_{\min} = \{ \Phi\in\dom D_{\max} \mid \mathcal R\Phi = 0 \}$.
In particular, $\mathcal R$ induces an isomorphism
\begin{equation*}
  \check H(A) \cong \dom D_{\max}/\dom D_{\min} ;
\end{equation*}
\item\label{boucon}
for any closed subspace $B\subset\check H(A)$, the operator $D_{B,\max}$
with domain 
\begin{equation*}
  \dom D_{B,\max} = \{ \Phi\in\dom D_{\max} \mid \mathcal R\Phi \in B \}
\end{equation*}
is a closed extension of $D$ between $D_{\min}$ and $D_{\max}$,
and any closed extension of $D$ between $D_{\min}$ and $D_{\max}$ is of this form;
\item\label{parintmd}
for all $\Phi\in\dom D_{\max}$ and $\Psi\in\dom D^*_{\max}$,
\begin{equation*}
  \int_M \langle D_{\max}\Phi,\Psi\rangle\dV 
  = \int_M \langle\Phi,D^*_{\max}\Psi\rangle\dV
  - \int_{\partial M} \langle \sigma_D(\nu^\flat)\mathcal R\Phi,\mathcal R\Psi \rangle\dS .
\tag{$\blacksquare$}
\end{equation*}
\end{asparaenum}
\end{thm}

\usetagform{default}

\begin{remark}
As a topological vector space, $\check H(A)$ does not depend on the choice
of adapted operator $A$, by \tref{maxdom}.\ref{maxmin}.
The pairing in \tref{maxdom}.\ref{parintmd} is well defined because $\sigma_D(\nu^\flat)$ maps $\check H(A)$ to $\hat H(A)$ by \eqref{atilde}.
\end{remark}

\usetagform{simple}

\begin{thm}[{Boundary regularity I, \cite[Thm.~6.11]{BB}}]\label{regbou}
Assume the \setup.
Let $k\ge0$ be an integer and $\Phi\in\dom D_{\max}$.
Then
\begin{equation*}
 \Phi \in H^{k+1}_{\loc}(M,E)
\Longleftrightarrow
D\Phi\in H^{k}_{\loc}(M,F) \mbox{ and } Q_{[0,\infty)}\mathcal R\Phi\in H^{k+1/2}(\partial M,E).
\end{equation*}
In particular, 
\begin{equation*}
 \Phi\in H^1_{\loc}(M,E) \Longleftrightarrow Q_{[0,\infty)}\mathcal R\Phi\in H^{1/2}(\partial M,E).
\tag{$\blacksquare$}
\end{equation*}
\end{thm}

\usetagform{default}

Note that $Q_{[0,\infty)}\mathcal R\Phi\in H^{1/2}(\partial M,E)$
if and only if $\mathcal R\Phi\in H^{1/2}(\partial M,E)$,
by \eqref{checka} and \tref{maxdom}.\ref{trace}.

\subsection{Boundary conditions}

\tref{maxdom}.\ref{boucon} justifies the following

\begin{definition}
A {\em boundary condition} for $D$ is a closed subspace of $\check H(A)$.
\end{definition}

In the notation of \tref{maxdom}.\ref{maxmin},
we write $D_{B,\max}$ for the operator with boundary values in a boundary condition $B$.
This differs from the notation of Atiyah-Patodi-Singer and others,
who would use a projection $P$ with $\ker P=B$ to write $P\mathcal R\Phi=0$.

\begin{thm}[{The adjoint operator, \cite[Sec.~7.2]{BB}}]\label{prodad}
Assume the \setup{} and that $B\subset\check H(A)$ is a boundary condition.
Let $\tilde A$ be adapted to $D^*$.
Then
\begin{equation*}
  \Bad := \{ \psi\in\check H(\tilde A) 
  \mid \text{$(\sigma _D(\nu^\flat)\phi,\psi)=0$, for all $\phi\in B$} \}
\end{equation*}
is a closed subspace of $\check H(\tilde A)$, that is, it is a boundary condition for $D^*$.
Moreover, the adjoint operator of $D_{B,\max}$ is the operator $D^*_{\Bad,\max}$.
\hfill$\blacksquare$
\end{thm}

\subsection{\emph{D}-elliptic boundary conditions}

For $V\subset H^{-\infty}(\partial M,E)$ and $s\in\R$, let
\begin{equation*}
  V^s := V \cap H^s(\partial M,E)) .
\end{equation*}
For subspaces $V,W\subset L^2(\partial M,E)$,
we say that a bounded linear operator $g:V\to W$ is of {\em order zero} if
\begin{equation*}
  g(V^s) \subset W^s ,
\end{equation*}
for all $s\ge0$.
For example, spectral projections $Q_I$ as in \eqref{specpro} are of order zero.

\begin{definition}\label{dell}
A linear subspace $B\subset H^{1/2}(\partial M,E)$
is said to be a {\em $D$-elliptic boundary condition}
if there is an $L^2$-orthogonal decomposition
\begin{equation}
  L^2(\dM,E) = V_- \oplus W_- \oplus V_+ \oplus W_+
  \label{elldec}
\end{equation}
such that
\begin{equation*}
  B = W_+ \oplus \{ v+gv \mid v \in V_-^{1/2} \} ,
\end{equation*}
where\\
\begin{inparaenum}[1)]
\item\label{ebcw}
$W_-$ and $W_+$ are finite dimensional and contained in $C^\infty(\dM,E)$;\\
\item\label{ebcv}
$V_- \oplus W_-\subset L^2_{(-\infty,a]}(A)$
and $V_+ \oplus W_+\subset L^2_{[-a,\infty)}(A)$, for some $a\in\R$;\\
\item\label{ebcg}
$g:V_-\to V_+$  and its adjoint $g^*:V_+\to V_-$ are operators of order 0.
\end{inparaenum}
\end{definition}

\begin{remarks}\label{remell}
\begin{inparaenum}[1)]
\item
$D$-elliptic boundary conditions are closed in $\check H(A)$,
and hence they are boundary conditions in the sense formulated further up.\\
\item\label{specialdecomposition}
If $B$ is a $D$-elliptic boundary condition and $a\in\R$ is given,
then the decomposition \eqref{elldec} can be chosen such that
\begin{equation*}
  V_- \oplus W_- = L^2_{(-\infty,a)}(\partial M,E)
  \quad\text{and}\quad
  V_+ \oplus W_+ = L^2_{[a,\infty)}(\partial M,E) .
\end{equation*}
\item
If $B$ is a $D$-elliptic boundary condition, then $\Bad$ is $D^*$-elliptic.
In fact, using $\tilde A$ as in \eqref{atilde}, we get
\begin{equation}
  \Bad = \sigma_D(\nu^\flat) \big(W_- \oplus \{ v-g^*v \mid v \in V_+^{1/2} \} \big) .
  \label{bad}
\end{equation}
\end{inparaenum}
\end{remarks}

The point of $D$-ellipticity of a boundary condition is that it ensures good regularity properties up to boundary just like ellipticity of the differential operator implies good regularity properties in the interior.

\begin{thm}[{Boundary regularity II, \cite[Thm.~7.17]{BB}}]\label{regula}
Assume the \setup{} and that $B\subset\check H(A)$ is a $D$-elliptic boundary condition.
Then
\begin{align*}
  \Phi\in H^{k+1}_{\loc}(M,E) &\Longleftrightarrow
  D_{B,\max}\Phi\in H^{k}_{\loc}(M,F),
\end{align*}
for all $\Phi\in\dom D_{B,\max}$ and integers $k\ge0$.
In particular,
$\Phi\in\dom D_{B,\max}$ is smooth up to the boundary
if and only if $D\Phi$ is smooth up to the boundary.
\hfill$\blacksquare$
\end{thm}

\begin{thm}\label{smoothdense}
Assume the \setup{} and that $B\subset\check H(A)$ is a $D$-elliptic boundary condition.
Then
\begin{equation*}
  C^\infty_c(M,E;B)
  := \{ \Phi \in C^\infty_c(M,E) \mid \mathcal R(\Phi) \in B \}
\end{equation*}
is dense in $\dom D_{B,\max}$ with respect to the graph norm.
\end{thm}

\begin{proof}
Choose a representation of $B$ as in Remark \ref{remell}.\ref{specialdecomposition}.
Since $W_-$ is finite dimensional and contained in $C^\infty(\dM,E)$,
we get that $V_-\cap C^\infty(\dM,E)$ is dense in $V_-$,
and similarly for $V_+$.
Since $g$ is of order $0$, we conclude that
\[
  \{ v+gv \mid v \in V_-^{1/2} \}\cap C^\infty(\dM,E)
\]
is dense in $\{ v+gv \mid v \in V_-^{1/2} \}$.
Hence $B\cap C^\infty(\dM,E)$ is dense in $B$.

Let $\Phi\in\dom D_{B,\max}$ and set $\varphi:=\mathcal R\Phi$.
Choose an extension operator $\mathcal E$ as in (43) in \cite{BB}.
Then $\Psi:=\Phi-\mathcal E\varphi$ vanishes along $\dM$,
and hence $\Psi\in\dom D_{\min}$, by \tref{maxdom}.\ref{maxmin}.
Therefore $\Psi$ is the limit of smooth sections in $C^\infty_{cc}(M,E)$,
by the definition of $D_{\min}$.

It remains to show that $\mathcal E\varphi$ can be approximated by
smooth sections in $C^\infty(M,E;B)$.
As explained in the beginning of the proof,
there is a sequence $(\varphi_n)$ in $B\cap C^\infty(\dM,E)$ converging to $\varphi$.
Then $\mathcal E\varphi_n\in C^\infty(M,E;B)$
and $\mathcal E\varphi_n\to\mathcal E\varphi$ with respect to the graph norm,
by Lemma 5.5 in \cite{BB}.
\end{proof}

\subsection{Self-adjoint \emph{D}-elliptic boundary conditions}

Assume the \setup, that $E=F$ and that $D$ is formally self-adjoint.
Choose $\tilde A$ as in \eqref{atilde}.
Let $B\subset H^{1/2}(\partial M,E)$ be a $D$-elliptic boundary condition.
Then $D_{\Bad,\max}$ is the adjoint operator of $D_{B,\max}$,
where $\Bad$ is given by \eqref{bad}.
In particular, $D_{B,\max}$ is self-adjoint if and only if $B$ is self-adjoint,
that is, if and only if $B=\Bad$.

Note that $\Bad$ is the image of the $L^2$-orthogonal complement
of $B$ in $H^{1/2}(\partial M,E)$ under $\sigma_D(\nu^\flat)$.
Hence $B=\Bad$ if and only if $\sigma_D(\nu^\flat)$ interchanges $B$
with its $L^2$-orthogonal complement in $H^{1/2}(\partial M,E)$.

\begin{thm}
Assume the \setup, that $E=F$ and that $D$ is formally self-adjoint.
Let $B$ be a self-adjoint $D$-elliptic boundary condition.

Then $D$ is essentially self-adjoint on
\begin{equation*}
  C^\infty_{c}(M,E;B) = \{ \Phi \in C^\infty_c(M,E) \mid \mathcal R\Phi \in B \} ,
\end{equation*}
and the closure of $D$ on $C^\infty_{c}(M,E;B)$ is $D_{B,\max}$.
\end{thm}

\begin{proof}
By \tref{smoothdense}, $C^\infty_{c}(M,E;B)$ is dense in $\dom D_{B,\max}$.
\end{proof}

The following result adapts and extends Theorem 1.83 in \cite{BBC}
to $D$-elliptic boundary conditions as considered here.

\usetagform{simple}

\begin{thm}[Normal form for $B$]\label{normb}
Assume the \setup, that $E=F$ and that $D$ is formally self-adjoint. 
Suppose that $\sigma_D(\nu^\flat)$ anticommutes with $A$.
Then a $D$-elliptic boundary condition $B$ is self-adjoint if and only if there is

\begin{inparaenum}[1)]
\item\label{normbvw}
an orthogonal decomposition 
$L^2_{(-\infty,0)}(A)=V\oplus W,$
where $W$ is a finite dimensional subspace of $C^\infty(\partial M,E)$,\\
\item\label{normbl} 
an orthogonal decomposition $\ker A=L\oplus\sigma_D(\nu^\flat)L$,\\
\item\label{normbg}
and a self-adjoint operator $g:V\oplus L\to V\oplus L$ of order zero
\end{inparaenum}
such that
\begin{equation*}
  B = \sigma_D(\nu^\flat)W \oplus \{ v + \sigma_D(\nu^\flat)gv \mid v \in V^{1/2}\oplus L \} .
\tag{$\blacksquare$}
\end{equation*}
\end{thm}

\usetagform{default}

\begin{remarks}\label{remnormb}
\begin{inparaenum}[1)]
\item
In \tref{normb}, the case $\ker A=\{0\}$ is not excluded.
In this latter case, the representation of $B$ as in \tref{normb} is unique
since $V=Q_{(-\infty,0)}B$ and $W$ is the orthogonal complement of $V$
in $L^2_{(-\infty,0)}(A)$. \\
\item\label{sebcexist}
\tref{normb}.\ref{normbl} excludes the existence of self-adjoint boundary conditions
in the case where $\ker A$ is of odd dimension.
Conversely, if $\dim\ker A$ is even and the eigenvalues $i$ and $-i$ of $\sigma_D(\nu^\flat)$ have equal multiplicity, 
then self-adjoint boundary conditions exist.
A simple example is $H^{1/2}_{(-\infty,0)}(A)\oplus L$,
where $L$ is a subspace of $\ker A$ as in \tref{normb}.\ref{normbl}. \\
\item
Let $E$, $D$, and $A$ be the complexification of a Riemannian vector bundle,
a formally self-adjoint real Dirac-type operator,
and a real boundary operator $A_\R$, respectively.
Then $\sigma_D(\nu^\flat)$ turns the real kernel $\ker(A_\R)$ into a symplectic vector space.
It follows that the complexification $L$ of any Lagrangian subspace of $\ker(A_\R)$
will satisfy $\ker A=L\oplus\sigma_D(\nu^\flat)L$,
and hence self-adjoint elliptic boundary conditions exist, by the previous remark.\\
\item
First attempts have been made to relax the condition of compactness of $\dM$.
The results in \cite{GN} apply to the Dirac operator associated with a spin$^c$ structure when $M$ and $\dM$ are complete and geometrically bounded in a suitable sense.
\end{inparaenum}
\end{remarks}

\subsection{Local and pseudo-local boundary conditions}

Throughout this section,
we let $M$ be a complete Riemannian manifold with compact boundary, 
$E$ and $F$ be Hermitian vector bundles over $M$,
and $D$ be a Dirac-type operator from $E$ to $F$.

\begin{definition}
We say that a linear subspace $B \subset H^{1/2}(\dM,E)$ is a \emph{local boundary condition}
if there is a (smooth) subbundle $E'\subset E_{|\dM}$ such that
\begin{equation*}
  B = H^{1/2}(\dM,E').
\end{equation*}
More generally, we say that $B$ is \emph{pseudo-local} if there is
a classical pseudo-differential operator $P$ of order $0$ acting on sections of $E$
over $\dM$ which induces an orthogonal projection on $L^2(\dM,E)$ such that
\begin{equation*}
  B=P(H^{1/2}(\dM,E)) .
\end{equation*}
\end{definition}

\begin{thm}[{Characterization of pseudo-local boundary conditions, \cite[Thm.~7.20]{BB}}]\label{pseudo}
Assume the \setup.
Let $P$ be a classical pseudo-differential operator of order zero,
acting on sections of $E$ over $\partial M$.
Suppose that $P$ induces an orthogonal projection in $L^2(\dM,E)$.
Then the following are equivalent:
\begin{asparaenum}[(i)]
\item\label{lbcell}
$B=P(H^{1/2}(\dM,E))$ is a $D$-elliptic boundary condition.
\item\label{lbcfop} 
For some (and then all) $a\in\R$,
$$P-Q_{[a,\infty)}:L^2(\dM,E)\to L^2(\dM,E)$$
is a Fredholm operator. 
\item\label{lbcpdo} 
For some (and then all) $a\in\R$,
$$P-Q_{[a,\infty)}:L^2(\dM,E)\to L^2(\dM,E)$$
is an elliptic classical pseudo-differential operator of order zero.
\item\label{lbcloc}
For all $\xi\in T^*_x\dM\setminus\{0\}$, $x\in\dM$, 
the principal symbol $\sigma_P(\xi):E_x\to E_x$ restricts to an isomorphism
from the  sum of the eigenspaces for the negative eigenvalues of $i\sigma_{A}(\xi)$
onto its image $\sigma_P(\xi)(E_x)$.
\hfill$\blacksquare$
\end{asparaenum}
\end{thm}

\begin{remark}
The projection $P$ is closely related to the Calder\'on projector $\P$ studied in the literature,
see e.g.\,\cite{BW}.
If the Calder\'on projector is chosen self-adjoint as described in \cite[Lemma~12.8]{BW},
then $P=\id-\P$ satisfies the conditions in \tref{pseudo}.
\end{remark}

Our concept of $D$-elliptic boundary conditions covers in particular
that of classical elliptic boundary conditions in the sense of Lopatinski and Shapiro.

\begin{cor}[{\cite[Cor.~7.22]{BB}}]\label{lopashap}
Let $E' \subset E_{|\dM}$ be a subbundle
and $P:E_{|\dM}\to E'$ be the fiberwise orthogonal projection.
If $(D,\id-P)$ is an elliptic boundary value problem in the classical sense
of Lopatinski and Shapiro, then $B=H^{1/2}(\dM,E')$
is a local $D$-elliptic boundary condition.
\end{cor}

\begin{proof}
Let $(D,\id-P)$ be an elliptic boundary value problem in the classical sense
of Lopatinsky and Shapiro, see e.~g.\ \cite[Sec.~1.9]{Gi}.
This means that the rank of $E'$ is half of that of $E$ and that, for any $x\in\dM$,
any $\eta\in T^*_x\dM \setminus \{0\}$, and any $\phi\in (E'_x)^\perp$,
there is a unique solution $f:[0,\infty)\to E_x$ to the ordinary differential equation
\begin{equation}
\left(i\sigma_{A}(\eta) +\frac{d}{dt}\right)f(t)=0
\label{eq:LopShap}
\end{equation}
subject to the boundary conditions
$$
(\id-P)f(0)=\phi
\quad\mbox{ and }\quad
\lim_{t\to\infty}f(t)=0 .
$$
Recall from Subsection~\ref{subsec:adapted} that $i\sigma_{A}(\eta)$ is Hermitian,
hence diagonalizable with real eigenvalues.
The solution to \eqref{eq:LopShap} is given by $f(t)= \exp(-it\sigma_{A}(\eta))\phi$.
The condition $\lim_{t\to\infty}f(t)=0$ is therefore equivalent to $\phi$
lying in the sum of the eigenspaces to the positive eigenvalues of $i\sigma_{A}(\eta)$.
This shows criterion \eqref{lbcloc} of \tref{pseudo}.
\end{proof}

As a direct consequence of \tref{pseudo} \eqref{lbcloc} we obtain

\begin{cor}\label{cor:involution}
Let $E_{|\dM}=E'\oplus E''$ be a decomposition such that $\sigma_{A}(\xi)=\sigma_D(\nu^\flat)^{-1}\sigma_D(\xi)$ interchanges $E'$ and $E''$, for all $\xi\in T^*\dM$.
Then $B':=H^{1/2}(\dM,E')$ and $B'':=H^{1/2}(\dM,E'')$ are local $D$-elliptic boundary conditions.
\qed
\end{cor}

This corollary applies, in particular, if $A$ itself interchanges sections of $E'$ and $E''$.

\subsection{Examples}
\label{susexas}

In this section, we discuss some important elliptic boundary conditions.

\begin{example}[Differential forms]
Let 
\[
E= \bigoplus_{j=0}^n \Lambda^jT^*M = \Lambda^*T^*M 
\] 
be the sum of the bundles of $\C$-valued alternating forms over $M$.
The Dirac-type operator is given by $D=d+d^*$,
where $d$ denotes exterior differentiation.

As before, $\nu$ is the interior unit normal vector field along the boundary $\dM$
and $\nu^\flat$ the associated unit conormal one-form.
For each $x\in\dM$ and $0\le j\le n$, we have a canonical identification 
\[
\Lambda^jT^*_xM
= \big(\Lambda^jT^*_x\dM\big) \oplus \big(\nu^\flat(x)\wedge\Lambda^{j-1}T^*_x\dM\big), 
\quad
\phi = \phi^{\tan} + \nu^\flat\wedge \phi^\nor.
\]
The local boundary condition corresponding to the subbundle
$E':=\Lambda^*\dM \subset E_{|\dM}$ is called the {\em absolute boundary condition},
\[
B_\abs 
= 
\{\phi\in H^{1/2}(\dM,E) \mid \phi^\nor =0 \} ,
\]
while $E'':=\nu^\flat\wedge\Lambda^*\dM \subset E_{|\dM}$
yields the {\em relative boundary condition},
\[
B_\rel
=
\{\phi\in H^{1/2}(\dM,E) \mid \phi^{\tan} =0 \} .
\]
Both boundary conditions are known to be elliptic in the classical sense of Lopatinski and Shapiro,
see e.g.\,\cite[Lemma~4.1.1]{Gi}.
Indeed, for any $\xi\in T^*\dM$,
the symbol $\sigma_D(\xi)$ leaves the subbundles $E'$ and $E''$ invariant,
while $\sigma_D(\nu^\flat)$ interchanges them.
Hence $\sigma_{A}(\xi)$ interchanges $E'$ and $E''$.
By \cref{cor:involution}, both, the absolute and the relative boundary condition,
are local $D$-elliptic boundary conditions.
\end{example}

\begin{example}[Boundary chirality]\label{chirco1}
Let $\chi$ be an orthogonal involution of $E$ along $\dM$
and denote by $E|_\dM=E^+\oplus E^-$ the orthogonal splitting
into the eigenbundles of $\chi$ for the eigenvalues $\pm1$.
We say that $\chi$ is a \emph{boundary chirality} (with respect to $A$)
if $\chi$ anticommutes with $A$.
The associated boundary conditions \mbox{$B_{\pm\chi} = H^{1/2}(\dM,E^\pm)$} are $D$-elliptic, by \cref{cor:involution}.
In fact, \mbox{$\chi H^{1/2}_{(-\infty,0)}(A)=H^{1/2}_{(0,\infty)}(A)$}
since $\chi$ anticommutes with $A$, and hence
\begin{equation*}
  B_{\pm\chi}
  = \{ \phi \in \ker A \mid \chi\phi = \pm\phi \} 
  \oplus \{ \phi \pm \chi\phi \mid \phi \in H^{1/2}_{(-\infty,0)}(A) \} .
\end{equation*}
We have $B_{-\chi}=B_{\chi}^\perp$ and hence $\sigma_D(\nu^\flat)B_{-\chi}$
is the adjoint of $B_\chi$.

An example of a boundary chirality is $\chi=i\sigma_D(\nu^\flat)$
in the case where $D$ is formally self-adjoint and $A$ has been chosen to anticommute with $\chi$ as in \lref{norma}.
This occurs, for instance, if $D$ is a Dirac operator in the sense of Gromov and Lawson
and $A$ is the canonical boundary operator for $D$; see Appendix~\ref{sec:GromovLawson}.

There is a refinement which is due to Freed \cite[\S2]{Fr}:
enumerate the connected components of $\dM$ as $N_1, \ldots, N_k$
and associate a sign $\eps_j\in\{-1,1\}$ to each component $N_j$.
Then
\begin{equation*}
  \chi\phi := \sum\nolimits_j i\eps_j\sigma_D(\nu^\flat)\phi_j ,
\end{equation*}
where $\phi_j:=\phi_j|_{N_j}$, is again a boundary chirality.
It has the additional property that it commutes with $i\sigma_D(\nu^\flat)$;
compare \lref{freedex} and \tref{freedb}.
\end{example}

\begin{example}[Generalized Atiyah-Patodi-Singer boundary conditions]\label{gaps}
Let $D$ be a Dirac-type operator and $A$ an admissible boundary operator.
Fix $a\in\R$ and let
\begin{equation*}
  V_- := L^2_{(-\infty,a)}(A) , \quad
  V_+ := L^2_{[a,\infty)}(A) , \quad
  W_- =W_+ := \{0\} ,
  \quad\text{and}\quad g = 0 .
\end{equation*}
Then the $D$-elliptic boundary condition
\[
  B(a) = H_{(-\infty,a)}^{1/2}(A) .
\]
is known as a \emph{generalized Atiyah-Patodi-Singer boundary condition}.
The (nongeneralized) Atiyah-Patodi-Singer boundary condition
as studied in \cite{APS} is the special case $a=0$.
Generalized APS boundary conditions are not local.
However, they are still pseudo-local,
by \cite[p.~48]{APS} together with \cite{Se} or by \cite[Prop.~14.2]{BW}.
\end{example}

\begin{example}[Modified Atiyah-Patodi-Singer boundary conditions]\label{maps}
The modified APS boundary condition, introduced in \cite{HMR}, is given by
\begin{equation*}
  \BmAPS
  = \{ \phi \in H^{1/2}(\partial M,E) \mid
  \phi + \sigma_D(\nu^\flat)\phi \in H^{1/2}_{(-\infty,0)}(A) \} .
\end{equation*}
It requires that the spectral parts $\phi=\phi_{(-\infty,0)}+\phi_{0}+\phi_{(0,\infty)}$
of $\phi\in \BmAPS$ satisfy
\begin{equation*}
  \phi_{(0,\infty)} = - \sigma_D(\nu^\flat)\phi_{(-\infty,0)} 
  \quad\text{and}\quad
  \phi_{0} = - \sigma_D(\nu^\flat)\phi_{0} .
\end{equation*}
Since $\sigma_D(\nu^\flat)^2=-1$, we get $\phi_0=0$.
Thus $\BmAPS$ is $D$-elliptic with the choices
\[
  V_- = L^2_{(-\infty,0)}(A) , \,
  V_+ = L^2_{(0,\infty)}(A) , \,
  W_- = \ker(A) , \,
  W_+ = \{0\} ,
  \,\text{and}\, g = -\sigma_D(\nu^\flat) .
\]
\end{example}

\begin{example}[Transmission conditions]\label{ex:TransCond}
Let $M$ be a complete Riemannian manifold.
For the sake of simplicity, assume that the boundary of $M$ is empty, even though this is not really necessary.
Let $N\subset M$ be a compact hypersurface with trivial normal bundle.
Cut $M$ along $N$ to obtain a Riemannian manifold $M'$ with compact boundary.
The boundary $\dM'$ consists of two copies $N_1$ and $N_2$ of $N$.
We may write $M' = (M\setminus N)\sqcup N_1 \sqcup N_2$ (Fig.~\ref{fig:transmission}).

Let $E,F\to M$ be Hermitian vector bundles
and $D$ be a Dirac-type operator from $E$ to $F$.
We get induced bundles $E'\to M'$ and $F'\to M'$
and a Dirac-type operator $D'$ from $E'$ to $F'$.
For $\Phi\in H^1_{\loc}(M,E)$,
we get $\Phi'\in H^1_{\loc}(M',E')$ such that $\Phi'|_{N_1}=\Phi'|_{N_2}$.
We use this as a boundary condition for $D'$ on $M'$.
We set
\[
  B :=
  \left\{(\phi,\phi)\in H^{1/2}(N_1,E)\oplus H^{1/2}(N_2,E) \mid \phi\in H^{1/2}(N,E)\right\} ,
\]
\begin{figure}[h]
\begin{pspicture}(0,-5.3)(12.721025,1.1710808)
\psset{unit=7mm}
\begin{psclip}{\pspolygon[linecolor=white](-0.5,-7.5)(9,-7.5)(9,3)(-0.5,3)}
\psbezier[linewidth=0.04](12.501025,2.1510808)(11.741025,1.8310809)(6.443285,2.1186054)(5.4610248,2.111081)(4.478765,2.1035564)(0.0,0.99085987)(0.021024996,-0.008919082)(0.042049993,-1.008698)(5.102699,-2.1510808)(6.161025,-2.1089191)(7.219351,-2.0667572)(11.741025,-0.46891907)(12.701025,-0.68891907)
\psbezier[linewidth=0.04](12.381025,1.6910809)(12.041025,1.6110809)(7.0244503,1.4894793)(7.081025,0.4910809)(7.1376,-0.5073174)(11.701025,-0.028919082)(12.541025,-0.108919084)
\psbezier[linewidth=0.04](3.101025,0.4910809)(3.101025,-0.30891907)(5.621025,-0.3889191)(5.621025,0.41108093)
\psbezier[linewidth=0.04](3.301025,0.16832179)(3.661025,0.7910809)(4.9610248,0.7310809)(5.2583065,-0.028919082)
\psbezier[linewidth=0.06](4.021025,-0.108919084)(3.301025,-0.008919082)(3.341025,-1.8089191)(3.811025,-1.6756418)
\psbezier[linewidth=0.06,linestyle=dashed,dash=0.16cm 0.16cm](3.925706,-0.08891908)(4.781025,-0.39566067)(4.2703867,-1.9089191)(3.781025,-1.6839752)
\rput(6.302431,1.2360809){$M$}
\rput(3.1524312,-0.8439191){$N$}

\rput(0,-5){
\psbezier[linewidth=0.04](12.501025,2.1510808)(11.741025,1.8310809)(6.443285,2.1186054)(5.4610248,2.111081)(4.478765,2.1035564)(0.0,0.99085987)(0.021024996,-0.008919082)(0.042049993,-1.008698)(5.102699,-2.1510808)(6.161025,-2.1089191)(7.219351,-2.0667572)(11.741025,-0.46891907)(12.701025,-0.68891907)
\psbezier[linewidth=0.04](12.381025,1.6910809)(12.041025,1.6110809)(7.0244503,1.4894793)(7.081025,0.4910809)(7.1376,-0.5073174)(11.701025,-0.028919082)(12.541025,-0.108919084)
\psbezier[linewidth=0.04](3.101025,0.4910809)(3.101025,-0.30891907)(5.621025,-0.3889191)(5.621025,0.41108093)
\psbezier[linewidth=0.04](3.301025,0.16832179)(3.661025,0.7910809)(4.9610248,0.7310809)(5.2583065,-0.028919082)

\psdots[dotsize=0.2,linecolor=white](3.65,-1.7)(3.85,-0.1)

\psbezier[linewidth=0.06](4.021025,-0.108919084)(3.301025,-0.008919082)(3.341025,-1.8089191)(3.811025,-1.6756418)
\psbezier[linewidth=0.06,linestyle=dashed,dash=0.16cm 0.16cm](3.925706,-0.08891908)(4.781025,-0.39566067)(4.2703867,-1.9089191)(3.781025,-1.6839752)
\psbezier[linewidth=0.06](3.761025,-0.068919085)(3.041025,0.031080918)(3.081025,-1.7689191)(3.551025,-1.6356418)
\psbezier[linewidth=0.06,linestyle=dotted,dash=0.16cm 0.16cm](3.685706,-0.04891908)(4.541025,-0.35566065)(4.030387,-1.8689191)(3.541025,-1.6439753)

\rput(6.302431,1.2360809){$M'$}
\rput(4.0,-2.2){$N_2$}
\rput(3.3,-2.0){$N_1$}
}
\end{psclip}
\end{pspicture} 

\caption{Cutting $M$ along the hypersurface $N$}
\label{fig:transmission}
\end{figure}
where we identify
\[
  H^{1/2}(N_1,E)=H^{1/2}(N_2,E)=H^{1/2}(N,E) .
\]
Let $A=A_0 \oplus -A_0$ be an adapted boundary operator for $D'$.
Here $A_0$ is a self-adjoint Dirac-type operator on $\Cu(N,E)=\Cu(N_1,E')$
and similarly $-A_0$ on $\Cu(N,E)=\Cu(N_2,E')$.
The sign is due to the opposite relative orientations of $N_1$ and $N_2$ in $M'$.

To see that $B$ is a $D'$-elliptic boundary condition, put 
\begin{align*}
V_+ &:= L^2_{(0,\infty)}(A_0\oplus -A_0) 
      = L^2_{(0,\infty)}(A_0) \oplus  L^2_{(-\infty,0)}(A_0) , \\
V_- &:= L^2_{(-\infty,0)}(A_0\oplus -A_0) 
      = L^2_{(-\infty,0)}(A_0) \oplus L^2_{(0,\infty)}(A_0)  , \\
W_+ &:= \{(\phi,\phi)\in\ker(A_0)\oplus\ker(A_0)\}  , \\
W_- &:= \{(\phi,-\phi)\in\ker(A_0)\oplus\ker(A_0)\} ,
\end{align*}
and
\[
g:V_-^{1/2} \to V_+^{1/2}, \quad g=\begin{pmatrix}
                                   0 & \id \cr 
                                   \id & 0
                                   \end{pmatrix} .
\]
With these choices $B$ is of the form required in \dref{dell}.
We call these boundary conditions {\em transmission conditions}.
Transmission conditions are not pseudo-local.

If $M$ has a nonempty boundary and $N$ is disjoint from $\dM$,
let us assume that we are given a $D$-elliptic boundary condition for $\dM$.
Then the same discussion applies if one keeps the boundary condition on $\dM$
and extends $B$ to $\dM' = \dM \sqcup N_1 \sqcup N_2$ accordingly.
\end{example}

\section{Spectral theory}
\label{spec}

Throughout htis section we assume the \setup.

\subsection{Coercivity at infinity}

For spectral and index theory we will also need boundary conditions at infinity if $M$ is noncompact.
Such conditions go under the name coercivity at infinity.

\begin{definitions}
For $\kappa>0$, we say that $D$ is $\kappa$-{\em coercive at infinity}
if there is a compact subset $K\subset M$ such that
\begin{equation*}
  \kappa\| \Phi \|_{L^2(M)}
  \le \| D\Phi \|_{L^2(M)} ,
\end{equation*}
for all smooth sections $\Phi$ of $E$ with compact support in $M\setminus K$.
If $D$ is $\kappa$-coercive at infinity for some $\kappa>0$,
then we call $D$ {\em coercive at infinity}.
\end{definitions}
Boundary conditions are irrelevant for coercivity at infinity because the compact set $K$ can always be chosen such that it contains a neighborhood of $\dM$.

\begin{examples}
\begin{inparaenum}[1)]
\item
If $M$ is compact, then $D$ is $\kappa$-coercive at infinity, for any $\kappa>0$.
Simply choose $K=M$.\\
\item 
If $D$ is formally self-adjoint and, outside a compact subset $K\subset M$,
all eigenvalues of the endomorphism $\mathcal{K}$ in the Weitzenb\"ock formula
\eqref{eq:Weitzenboeck} are bounded below by a constant $\kappa>0$,
then we have, for all $\Phi\in\Cucc(M,E)$ with support disjoint from $K$,
\begin{equation*}
  \LzM{D\Phi}^2 = \LzM{\nabla\Phi}^2 + (\mathcal{K}\Phi,\Phi)_{L^2(M)} 
  \geq \kappa\LzM{\Phi}^2 .
\end{equation*}
Hence $D$ is $\sqrt{\kappa}$-coercive at infinity in this case.\\
\item
Let $M=S^n\times [0,\infty)$, endowed with the product metric $g_{0}+dt^2$,
where $g_0$ is the standard Riemannian metric of the unit sphere and $t$ is the standard coordinate on $[0,\infty)$.
Consider the usual Dirac operator $D$ acting on spinors,
and denote by $\nabla$ the Levi-Civita connection on the spinor bundle.
The Lichnerowicz formula gives
\begin{equation*}
  D^2 = \nabla^*\nabla + R/4 ,
\end{equation*}
where $R=n(n-1)$ is the scalar curvature of $M$ (and $S^n$).
It follows that $D$ is $\sqrt{n(n-1)}/2$-coercive at infinity.\\
\item
Consider the same manifold $M=S^n\times [0,\infty)$, but now equipped with the warped metric $e^{-2t}g_{0}+dt^2$.
The scalar curvature is easily computed to be
\begin{equation*}
  R = R(t) = n(n-1) e^{2t} - n(n+1) \to \infty .
\end{equation*}
It follows that this time the Dirac operator $D$
is $\kappa$-coercive at infinity, for any $\kappa>0$.
\end{inparaenum}
\end{examples}

\begin{thm}[{\cite[Thm.~8.5]{BB}}]\label{thmcoer}
Assume the \setup.
Then the following are equivalent:
\begin{asparaenum}[(i)]
\item 
$D$ is coercive at infinity;
\item
$D_{B,\max}: \dom D_{B,\max} \to L^2(M,F)$ has finite dimensional kernel and closed image for \emph{some} $D$-elliptic boundary condition $B$;
\item
$D_{B,\max}: \dom D_{B,\max} \to L^2(M,F)$ has finite dimensional kernel and closed image for \emph{all} $D$-elliptic boundary conditions $B$.
\end{asparaenum}
In particular, $D$ and $D^*$ are coercive at infinity
if and only if $D_{B,\max}$ and $D^*_{\Bad,\max}$ are Fredholm operators for some/all $D$-elliptic boundary conditions $B$.
\hfill$\blacksquare$
\end{thm}

Extending the notion of Fredholm operator,
we say that a closed operator $T$ between Banach spaces $X$ and $Y$
is a \emph{left-} or \emph{right-Fredholm operator}
if the image of $T$ is closed and, respectively,
the kernel or the cokernel of $T$ is of finite dimension.
We say that $T$ is a \emph{semi-Fredholm operator}
if it is a left- or right-Fredholm operator, compare \cite[Section IV.5.1]{Ka}.
In this terminology,
\tref{thmcoer} says that $D_{B,\max}$ is a left-Fredholm operator for some/all $B$
if and only if $D$ is coercive at infinity.
For more on this topic, see \cite[IV.4 and IV.5]{Ka},
\cite[Appendix A]{BBC}, and \cite[Appendix A]{BB}.

In the case $X=Y$, we get corresponding \emph{essential} parts of the spectrum of $T$,
compare \cite[Section IV.5.6]{Ka} (together with footnotes).
We let
\begin{equation*}
  \spec_{\rm ess} T \subset \spec_{\rm nlf} T
  \subset \spec_{\rm nf} T \subset \spec T
\end{equation*}
be the set of $\lambda\in\C$ such that $T-\lambda$ is not a semi-Fredholm operator,
not a left-Fredholm operator, not a Fredholm operator, and not an isomorphism
from $\dom T$ to $X$, respectively, where \emph{ess} stands for \emph{essential}.
In the case where $X$ is a Hilbert space and where $T$ is self-adjoint,
$\ker T=(\im T)^\perp$ and $\spec T\subset\R$
so that, in particular, $\spec_{\rm ess} T=\spec_{\rm nf}T$.
Moreover, in this case, $\spec T\setminus\spec_{\rm ess}T$
consists of eigenvalues with finite multiplicities,
see Remark 1.11 in \cite[Section X.1.2]{Ka}.

\begin{cor}\label{corcoer}
Assume the \setup{} and $E=F$.
Let $B\subset H^{1/2}(\dM,E)$ be a $D$-elliptic boundary condition.
Let $\kappa>0$ and assume that $D$ is $\kappa$-coercive at infinity.
Then
\begin{equation*}
   \{ z \in \C \mid |z| < \kappa \} \cap \spec_{\rm nlf} D_{B,\max} = \emptyset .
\end{equation*}
If $D$ and $D^*$ are $\kappa$-coercive at infinity, then
\begin{equation*}
   \{ z \in \C \mid |z| < \kappa \} \cap \spec_{\rm nf} D_{B,\max} = \emptyset .
\end{equation*}
\end{cor}

\begin{proof}
For any $z\in\C$,
the operators $D-z$ and $(D-z)^*=D^*-\bar z$ are of Dirac type
such that $(D-z)_{\max}=D_{\max}-z$ and $(D^*-\bar z)_{\max}=D^*_{\max}-\bar z$.
Moreover, $B$ is a $(D-z)$-elliptic and $\Bad$ a $(D^*-\bar z)$-elliptic boundary condition,
one the adjoint of the other.  
By the triangle inequality, if $D$ is $\kappa$-coercive and $|z| < \kappa$,
then $D-z$ is $(\kappa-|z|)$-coercive, and similarly for $D^*-\bar z$.
Thus \tref{thmcoer} applies.
\end{proof}

\usetagform{simple}

\begin{cor}\label{cor:essspecest}
Assume the \setup, that $E=F$, and that $D$ is formally self-adjoint.
Let $B\subset H^{1/2}(\dM,E)$ be a self-adjoint $D$-elliptic boundary condition.
If $D$ is $\kappa$-coercive at infinity for \emph{some} $\kappa>0$, 
then $D_{B,\max}$ is self-adjoint with
\begin{align*}
   (-\kappa,\kappa) \cap \spec_{\rm ess} D_{B,\max} &= \emptyset . 
\tag{\qed}
\end{align*}
\end{cor}

\usetagform{default}

\begin{cor}\label{cor:discrete}
Assume the \setup, that $E=F$, and that $D$ is formally self-adjoint.
Let $B\subset H^{1/2}(\dM,E)$ be a self-adjoint $D$-elliptic boundary condition.
If $D$ is $\kappa$-coercive at infinity for \emph{all} $\kappa>0$, 
then $D_{B,\max}$ is self-adjoint with
\[\spec_{\rm ess}D_{B,\max}=\emptyset .\]
In particular, the eigenspaces of $D$ are finite dimensional,
pairwise $L^2$-orthogonal,
and their sum spans $L^2(M,E)$ in the sense of Hilbert spaces.
Moreover, eigensections of $D$ are smooth on $M$ (up to the boundary). \qed
\end{cor}

\begin{remark}
If $M$ is compact,
then $D$ is $\kappa$-coercive at infinity for all $\kappa>0$.
Hence \cref{cor:discrete} applies if $M$ is compact with boundary.
On the other hand, the resolvent of $D_{B,\max}$ is compact in this case
so that the decomposition of $L^2(M,E)$ into finite dimensional eigenspaces
is also clear from this perspective.
\end{remark}

\subsection{Coercivity with respect to a boundary condition}
Now we discuss spectral gaps of $D$ about $0$.
We get interesting results for Dirac operators in the sense of Gromov and Lawson, see Appendix~\ref{sec:GromovLawson}.

\begin{definition}
For $\kappa>0$, we say that  $D$ is $\kappa$-{\em coercive with respect to a boundary condition} $B$ if
\begin{equation*}
  \kappa\| \Phi \|_{L^2(M)}
  \le \| D\Phi \|_{L^2(M)} ,
\end{equation*}
for all $\Phi\in C^\infty_c(M,E;B)$. 
\end{definition}

In contrast to coercivity at infinity, the boundary condition $B$ is now crucial for the concept of coercivity.

\usetagform{simple}

\begin{cor}\label{cor:GLcoercive}
Assume the \setup, that $E=F$, and that $D$ is formally self-adjoint.
Let $B\subset H^{1/2}(\dM,E)$ be a self-adjoint $D$-elliptic boundary condition.
If $D$ is $\kappa$-coercive with respect to $B$, for $\kappa>0$, 
then $D_{B,\max}$ is self-adjoint with
\[
(-\kappa,\kappa) \cap \spec D_{B,\max} = \emptyset.
\tag{\qed}
\]
\end{cor}

\usetagform{default}

\begin{thm}\label{thm:GLspec}
Assume the \setup{} with $E=F$ and that\\
\begin{inparaitem}[$\diamond$]
\item
$D$ is a Dirac operator in the sense of Gromov and Lawson;\\
\item
$B$ is a $D$-elliptic boundary condition;\\
\item
the canonical boundary operator $A:C^\infty(\dM,E)\to C^\infty(\dM,E)$ for $D$ satisfies $$\left((A-\tfrac{n-1}{2}H)\phi,\phi\right)\le0$$ for all $\phi\in B$,
where $H$ is the mean curvature $H$ along $\dM$ with respect to the interior unit normal vector field $\nu$;\\
\item
the endomorphism field $\mathcal{K}$ in the Weitzenb\"ock formula \eqref{eq:Weitzenboeck} satisfies $\mathcal{K}\ge \kappa >0.$
\end{inparaitem}
Then $D$ is $\sqrt{\tfrac{n\kappa}{n-1}}$-coercive with respect to $B$.
In particular, if $B$ is self-adjoint, then 
\[
\left(-\sqrt{\tfrac{n\kappa}{n-1}},\sqrt{\tfrac{n\kappa}{n-1}}\right) \cap \spec D_{B,\max} = \emptyset.
\]
\end{thm}

\begin{proof}
For any $\Phi\in C_c^\infty(M,E;B)$ we have by \eqref{eq:pf4} and \eqref{cauchydna}, again writing $\phi=\Phi|_\dM$,
\begin{align}
\tfrac{n-1}{n}\|D\Phi\|^2
&\ge
\int_M\langle\mathcal{K}\Phi,\Phi\rangle\dV
- \int_\dM (A-\tfrac{n-1}{2}H)|\phi|^2 \dS
\ge
\kappa \|\Phi\|^2.
\label{eq:eigest}
\end{align}
This proves $\sqrt{\tfrac{n\kappa}{n-1}}$-coerciveness with respect to $B$.
The statement on the spectrum now follows from \cref{cor:GLcoercive}.
\end{proof}

Here are some boundary conditions for which \tref{thm:GLspec} applies:

\begin{example}\label{ex:chiral}
Let $\chi$ be a boundary chirality with associated $D$-elliptic boundary condition $B_{\pm\chi} = H^{1/2}(\dM,E^\pm)$ as in \eref{chirco1}.
For $\phi,\psi\in B_{\chi}$, we have
\begin{equation*}
  ( A\phi,\psi )
  = (A\chi\phi,\psi )
  = - (\chi A\phi,\psi )
  = - ( A\phi,\chi\psi )
  = - ( A\phi,\psi ) .
\end{equation*}
Hence $(A\phi,\psi)=0$, for all $\phi,\psi\in B_{\chi}$,
and similarly for $B_{-\chi}$.
If $\chi$ anticommutes with $\sigma_D(\nu^\flat)$,
then $B_\chi$ and $B_{-\chi}$ are self-adjoint boundary conditions\footnote{If $\chi$ commutes with $\sigma_D(\nu^\flat)$,
then $B_\chi$ and $B_{-\chi}$ are adjoint to each other.}.
Hence \tref{thm:GLspec} applies if $H\ge0$.
In the case of the classical Dirac operator $D$ acting on spinors, this yields the eigenvalue estimate in \cite[Thm.~3]{HMR}.
\end{example}

\begin{example}\label{ex:APS}
The Atiyah-Patodi-Singer boundary condition
\begin{equation*}
 \BAPS = H^{1/2}_{(-\infty,0)}(A)
\end{equation*}
is $D$-elliptic with adjoint boundary condition
\begin{equation*}
  \BAPS^\ad = \BAPS\oplus\ker A .
\end{equation*}
Hence $D_{\BAPS,\max}$ is symmetric. 
If $\ker A$ is trivial, then $D_{B_{APS}}$ is self-adjoint and $\spec D_{B_{APS}}\subset\R$.
By definition of $\BAPS$, we have $(A\phi,\phi) \le -\mu_1\|\phi\|_{L^2(\dM)}^2$ for all $\phi\in\BAPS$ where $-\mu_1$ is the largest negative eigenvalue of $A$.
Hence \tref{thm:GLspec} applies if $H\ge-\tfrac{2}{n-1}\mu_1$.

In the case of the classical Dirac operator $D$ acting on spinors, this yields the eigenvalue estimate for the APS boundary condition in \cite[Thm.~2]{HMR}.
Note that the assumption $\ker A=0$ is missing in Theorem~2 of \cite{HMR}.
In fact, if $\ker A$ is nontrivial, then $D_{B_{APS}}$ is not self-adjoint and $\spec D_{B_{APS}}=\C$,
compare \cite[Section V.3.4]{Ka}.

If we can choose a subspace $L\subset \ker A$ as in \tref{normb}.\ref{normbl},
then $B=H^{1/2}_{(-\infty,0)}(A) \oplus L$ is a self-adjoint $D$-elliptic boundary condition.
We have $(A\phi,\phi)\le 0$ for all $\phi\in B$ and \tref{thm:GLspec} applies if $H\ge0$.
\end{example}

\begin{example}\label{ex:mAPS}
The modified APS boundary condition
\begin{equation*}
  \BmAPS
  = \{ \phi \in H^{1/2}(\partial M,E) \mid
  \phi + \sigma_D(\nu^\flat)\phi \in H^{1/2}_{(-\infty,0)}(A) \} 
\end{equation*}
as in \eref{maps} is $D$-elliptic with adjoint condition
\begin{equation*}
\begin{split}
  \BmAPS^\ad
  &= \{ \phi \in H^{1/2}(\partial M,E) \mid
  \phi_{(0,\infty)} = - \sigma_D(\nu^\flat)\phi_{(-\infty,0)} \} \\
  &= \BmAPS \oplus \ker A .
\end{split}
\end{equation*}
Hence $D_{\BmAPS,\max}$ is symmetric.
The remaining part of the discussion is as in the previous example,
except that we have $(A\phi,\psi)=0$, for all $\phi,\psi\in\BmAPS^\ad$.
In particular, \tref{thm:GLspec} applies if $\ker A=0$ and $H\ge0$.
In the case of the classical Dirac operator $D$ acting on spinors, this yields the eigenvalue estimate in \cite[Thm.~5]{HMR}. 
As in the case of the APS boundary condition, the requirement $\ker A=0$ needs to be added to the assumptions of Theorem~5 in \cite{HMR}. \end{example}

Next we discuss under which circumstances the ``extremal values'' $\pm\sqrt{\frac{n\kappa}{n-1}}$ actually belong to the spectrum.
For this purpose, we make the following

\begin{definition}
Let $D$ be a formally self-adjoint Dirac operator in the sense of Gromov and Lawson with associated connection $\nabla$.
A section $\Phi\in C^\infty(M,E)$ is called a \emph{$D$-Killing section} if 
\begin{equation}
\nabla_X\Phi = \alpha\cdot\sigma_D(X^\flat)^*\Phi
\label{eq:Killing}
\end{equation}
for some constant $\alpha\in\R$ and all $X\in TM$.
The constant $\alpha$ is called the \emph{Killing constant} of $\Phi$.
\end{definition}

\begin{remarks}\label{KillingEigenschaften}
\begin{inparaenum}[1)]
\item
If $D$ is the classical Dirac operator, then spinors satisfying \eqref{eq:Killing} are called Killing spinors.
This motivates the terminology.\\
\item 
Equation \eqref{eq:Killing} is overdetermined elliptic.
Hence the existence of a nontrivial solution imposes strong restrictions on the underlying geometry.
For instance, if a Riemannian spin manifold carries a nontrivial Killing spinor, it must be Einstein \cite[Thm.~B]{Fri}.
See \cite{B} for a classification of manifolds admitting Killing spinors.\\
\item\label{KillingEigensection}
Any $D$-Killing section with Killing constant $\alpha$ is an eigensection of $D$ for the eigenvalue $n\alpha$:
\begin{align*}
D\Phi
&=
\sum_{j=1}^n \sigma_D(e_j^\flat)\nabla_{e_j}\Phi 
=
\alpha\cdot\sum_{j=1}^n \sigma_D(e_j^\flat)\sigma_D(e_j^\flat)^*\Phi 
=
n\alpha\Phi.
\end{align*}
\item
Any $D$-Killing section satisfies the twistor equation \eqref{twist}:
\begin{align*}
\nabla_X\Phi 
&= 
\alpha\cdot\sigma_D(X^\flat)^*\Phi
=
\frac1n\cdot\sigma_D(X^\flat)^*D\Phi .
\end{align*}
\item
Since $\sigma_D(X^\flat)$ is skewhermitian, the connection $\hat\nabla_X=\nabla_X - \alpha\cdot\sigma_D(X^\flat)^*$ is also a metric connection.
Since $D$-Killing sections are precisely $\hat\nabla$-parallel sections, we conclude that any $D$-Killing section $\Phi$ has constant length $|\Phi|$.
\end{inparaenum}
\end{remarks}

\begin{thm}
In addition to the assumptions in \tref{thm:GLspec} assume that 
$M$ is compact and that the boundary condition $B$ is self-adjoint.

Then $\sqrt{\frac{n\kappa}{n-1}}\in\spec(D)$ or $-\sqrt{\frac{n\kappa}{n-1}}\in\spec(D)$ if and only if 
there is a nontrivial $D$-Killing section $\Phi$ with $\phi=\Phi|_\dM\in B$ and Killing constant $\sqrt{\frac{\kappa}{n(n-1)}}$ or $-\sqrt{\frac{\kappa}{n(n-1)}}$, respectively.
\end{thm}

\begin{proof}
Let $\sqrt{\frac{n\kappa}{n-1}}\in\spec(D)$, the case $-\sqrt{\frac{n\kappa}{n-1}}\in\spec(D)$ being treated similarly.
Since $M$ is compact, the spectral value $\sqrt{\frac{n\kappa}{n-1}}$ must be an eigenvalue by \cref{cor:essspecest}.
Let $\Phi$ be an eigensection of $D$ for the eigenvalue $\sqrt{\frac{n\kappa}{n-1}}$ satisfying the boundary condition.
Then we must have equality everywhere in the chain of inequalities \eqref{eq:eigest}.
In particular, $\Phi$ must solve the twistor equation \eqref{twist}.
Hence
\[
\nabla_X\Phi = \textstyle{\frac1{n}} \sigma_D(X^\flat)^* D\Phi = \sqrt{\frac{\kappa}{n(n-1)}}\sigma_D(X^\flat)^* \Phi .
\]
Conversely, if $\Phi$ is a $D$-Killing section with Killing constant $\sqrt{\frac{\kappa}{n(n-1)}}$, then $\Phi$ is an eigensection of $D$ for the eigenvalue $\sqrt{\frac{n\kappa}{n-1}}$, by Remark~\ref{KillingEigenschaften}.\ref{KillingEigensection}.
\end{proof}

\begin{example}
Let $M$ be the closed geodesic ball of radius $r\in (0,\pi)$ about $e_1$ in the unit sphere $S^n$.
The sectional curvature of $M$ is identically equal to $1$, its scalar curvature to $n(n-1)$.
The boundary $\dM$ is a round sphere of radius $\sin(r)$ (Fig.~\ref{fig:sphere}).
Its mean curvature with respect to the interior unit normal is given by $H= \cot(r)$.
\end{example}

We consider the classical Dirac operator acting on spinors.
The restriction of the spinor bundle to the boundary yields the spinor bundle of the boundary if $n$ is odd and the sum of two copies of the spinor bundle of the boundary if $n$ is even.
Accordingly, the canonical boundary operator is just the classical Dirac operator of the boundary if $n$ is odd and the direct sum of it and its negative if $n$ is even.
The kernel of the boundary operator is trivial.
\begin{figure}[h]
\begin{pspicture}(-2.5,-2.4)(3,2.5)
\psarc(0,0){2}{315}{225} 
\psline[linewidth=0.1ex,linestyle=dashed](0,2)(0,-1.4142)
\psline[linewidth=0.1ex](0,0)(1.4142,-1.4142)
\psline[linewidth=0.2ex](1.4142,-1.4142)(-1.4142,-1.4142)
\psarc[linewidth=0.1ex](0,0){0.4}{315}{90}

\psdots(0,0)(0,2)

\rput[b](0,2.1){$e_1$}
\rput[tr](0.6,-0.6){$1$}
\rput[b](0.6,-1.4){$\sin(r)$}
\rput[bl](0.1,0){r}
\end{pspicture}

\caption{Geodesic ball of radius $r$ in the unit sphere}
\label{fig:sphere}
\end{figure}

\tref{thm:GLspec} applies with all the boundary conditions described in Examples~\ref{ex:chiral}--\ref{ex:mAPS} if $r\le\frac{\pi}{2}$ because then $H\ge0$.
Therefore the spectrum of the Dirac operator on $M$ subject to any of these boundary conditions does not intersect $(-\frac{n}{2},\frac{n}{2})$.
The largest negative Dirac eigenvalue of the boundary is given by $-\mu_1=-\frac{n-1}{2\sin(r)}$.
Since we have 
\[
\tfrac{n-1}{2}H = \tfrac{n-1}{2}\cot(r) \ge  -\tfrac{n-1}{2}\sin(r) = -\mu_1, 
\]
\tref{thm:GLspec} applies in the case of APS boundary conditions (\eref{ex:APS}) for all $r\in(0,\pi)$.

The sphere $S^n$ and hence $M$ do possess nontrivial Killing spinors for both Killing constants $\pm\frac12$.
The restriction of such a Killing spinor to $\dM$ never satisfies the APS boundary conditions.
Thus the equality case in \tref{thm:GLspec} does not occur and $\pm\frac{n}{2}$ cannot lie in the spectrum of $D$ on $M$ subject to APS conditions.
Hence, under APS boundary conditions and for any $r\in(0,\pi)$, the spectrum of $D$ on $M$ does not intersect $[-\frac{n}{2},\frac{n}{2}]$.

The modified APS boundary conditions are satisfied by the restrictions of the Killing spinors only if $r=\frac{\pi}{2}$.
In this case, $\frac{n}{2}$ is an eigenvalue of $D$ on $M$.

\section{Index theory}
\label{inth}

Throughout this section, assume the \setup.
In \tref{thmcoer} we have seen that $D_{B,\max}: \dom D_{B,\max} \to L^2(M,F)$ is a Fredholm operator for any $D$-elliptic boundary condition provided $D$ and $D^*$ are coercive at infinity.
This is the case if $M$ is compact, for instance.
The index is the number
\[
  \ind D_{B,\max} = \dim\ker D_{B,\max} - \dim\ker D^*_{\Bad,\max} \in \Z .
\]
If $B$ is a $D$-elliptic boundary condition, then, by Theorems~\ref{maxdom}.\ref{boucon} and \ref{regbou},  $D_{B,\max}$ has domain
\begin{equation*}
  \dom D_{B,\max}
  = \{\Phi\in\dom D_{\max} \mid \mathcal R\Phi\in B\}
  \subset H^1_{\loc}(M,E) .
\end{equation*}
Since $\dom D_{B,\max}$ is contained in $H^1_{\loc}(M,E)$,
we will briefly write $D_B$ instead of $D_{B,\max}$.

\subsection{Fredholm property and index formulas}
\label{susefrepro}

As a direct consequence of \tref{thmcoer} we get

\begin{cor}[{\cite[Cor.~8.7]{BB}}]\label{corfred2}
Assume the \setup{} and that $D$ and $D^*$ are coercive at infinity.
Let $B$ be a $D$-elliptic boundary condition and let $\check C$ be a closed complement of $B$ in $\check H(A)$.
Let $\check P:\check H(A)\to\check H(A)$ be the projection with kernel $B$ and image $\check C$.
Then
\begin{equation*}
\check L: \dom D_{\max} \to L^2(M,F)\oplus\check C,
  \quad \check L\Phi = (D_{\max}\Phi,\check P\mathcal R\Phi),
\end{equation*}
is a Fredholm operator with the same index as $D_B$.\qed
\end{cor}

\usetagform{simple}

\begin{cor}[{\cite[Cor.~8.8]{BB}}]\label{cor:AgraDy}
Assume the \setup{} and that $D$ and $D^*$ are coercive at infinity.
Let $B_1 \subset B_2 \subset H^{1/2}(\dM,E)$ be $D$-elliptic boundary conditions for $D$.
Then $\dim(B_2/B_1)$ is finite and
\[
\ind(D_{B_2}) = \ind(D_{B_1}) + \dim(B_2/B_1).
\tag{$\blacksquare$}
\]
\end{cor}

\usetagform{default}

\begin{example}\label{exaAgDy}
For the generalized Atiyah-Patodi-Singer boundary conditions as in \eref{gaps}
and $a<b$, we have
\[
  \ind D_{B(b)} = \ind D_{B(a)} +  \dim L^2_{[a,b)}(A) .
\]
\end{example}

The following result says that index computations for $D$-elliptic boundary conditions
can be reduced to the case of generalized Atiyah-Patodi-Singer boundary conditions.

\usetagform{simple}

\begin{thm}[{\cite[Thm.~8.14]{BB}}]\label{thmad}
Assume the \setup{} and that $D$ and $D^*$ are coercive at infinity.
Let $B \subset H^{1/2}(\dM,E)$ be a $D$-elliptic boundary condition.
Then we have, in the representation of $B$ as in Remark~\ref{remell}.\ref{specialdecomposition},
\[
  \ind D_B = \ind D_{B(a)} + \dim W_+ - \dim W_- .
\]
\end{thm}

\usetagform{default}

\begin{proof}[Sketch of proof]
Replacing $g$ by $s g$, $s\in[0,1]$, yields a continuous $1$-parameter family
of $D$-elliptic boundary conditions.
One can show that the index stays constant under such a deformation of boundary conditions.
Therefore, we can assume without loss of generality that $g=0$,
i.e., $B=W_+ \oplus V_-^{1/2}$.
Consider one further boundary condition, 
\[
B' := W_- \oplus W_+ \oplus V_-^{1/2}
= H^{1/2}_{(-\infty,a)}(A) \oplus W_+ = B(a) \oplus W_+ .
\]
Applying \cref{cor:AgraDy} twice we conclude
\begin{align*}
\ind(D_B) 
&= 
\ind(D_{B'}) - \dim W_- 
=
\ind(D_{B(a)}) + \dim W_+ - \dim W_- . \qedhere
\end{align*}
\end{proof}

\subsection{Relative index theory}
\label{suserelind}

Assume the \setup{} throughout the section.
For convenience assume also that $M$ is connected and that $\dM=\emptyset$.
For what follows, compare Example~\ref{ex:TransCond}.
Let $N$ be a closed and two-sided hypersurface in $M$.
Cut $M$ along $N$ to obtain a manifold $M'$, possibly connected,
whose boundary $\dM'$ consists of two disjoint copies $N_1$ and $N_2$ of $N$,
see Figure~1 on page~\pageref{fig:transmission}.
There are natural pull-backs $E'$, $F'$, and $D'$
of $E$, $F$, and $D$ from $M$ to $M'$.
Choose an adapted operator $A$ for $D'$ along $N_1$.
Then $-A$ is an adapted operator for $D'$ along $N_2$
and will be used in what follows.

\begin{thm}[{Splitting Theorem, \cite[Thm.~8.17]{BB}}]\label{split}
For $M$, $M'$, and notation as above, $D$ and $D^*$ are coercive at infinity
if and only if $D'$ and $(D')^*$ are coercive at infinity.
In this case, $D$ and $D'_{B_1\oplus B_2}$ are Fredholm operators with
\[
  \ind D = \ind D'_{B_1\oplus B_2} ,
\]
where $B_1=B(a)=H^{1/2}_{(-\infty,a)}(A)$ and $B_2=H^{1/2}_{[a,\infty)}(A)$,
considered as boundary conditions along $N_1$ and $N_2$, respectively.
More generally,
we may choose any $D$-elliptic boundary condition $B_1\subset H^{1/2}(N,E)$
and its $L^2$-orthogonal complement $B_2\subset H^{1/2}(N,E)$.
\hfill$\blacksquare$
\end{thm}

Let $M_1$ and $M_2$ be complete Riemannian manifolds without boundary and
\begin{equation*}
  D_i: \Cu(M_i,E_i) \to \Cu(M_i,F_i)
\end{equation*}
be Dirac-type operators.
Let $K_1\subset M_1$ and $K_2\subset M_2$ be compact subsets.
Then we say that \emph{$D_1$ outside $K_1$ agrees with $D_2$ outside $K_2$}
if there are an isometry $f:M_1\setminus K_1 \to M_2\setminus K_2$
and smooth fiberwise linear isometries
\begin{equation*}
  \mathcal{I}_E:E_1|_{M_1\setminus K_1} \to E_2|_{M_2\setminus K_2}
  \quad\text{and}\quad
  \mathcal{I}_F:F_1|_{M_1\setminus K_1} \to F_2|_{M_2\setminus K_2}
\end{equation*}
such that
\begin{equation*}
\xymatrix{
E_1|_{M_1\setminus K_1} \ar[d] \ar[r]^{\mathcal{I}_E} & E_2|_{M_2\setminus K_2} \ar[d] && F_1|_{M_1\setminus K_1} \ar[d] \ar[r]^{\mathcal{I}_F}& F_2|_{M_2\setminus K_2} \ar[d]\\
M_1\setminus K_1 \ar[r]^f & M_2\setminus K_2 && M_1\setminus K_1\ar[r]^f & M_2\setminus K_2
}
\end{equation*}
commute and 
\begin{equation*}
  \mathcal{I}_F\circ (D_1\Phi) \circ f^{-1} = D_2(\mathcal{I}_E\circ\Phi\circ f^{-1})
\end{equation*}
for all smooth sections $\Phi$ of $E_1$ over $M_1\setminus K_1$.

Assume now that $D_1$ and $D_2$ agree outside compact domains $K_i\subset M_i$.
For $i=1,2$, choose a decomposition $M_i=M_i'\cup M_i''$ such that
$N_i=M_i'\cap M_i''$ is a compact hypersurface in $M_i$,
$K_i$ is contained in the interior of $M_i'$, $f(M_1'')=M_2''$, and $f(N_1)=N_2$. 
Denote the restriction of $D_i$ to $M_i'$ by $D_i'$.
The following result is a general version of the $\Phi$-relative index theorem
of Gromov and Lawson \cite[Thm.~4.35]{GL}.

\begin{thm}[{\cite[Thm.~1.21]{BB}}]\label{firelind}
Under the above assumptions,
let $B_1\subset H^{1/2}(N_1,E_1)$ and $B_2\subset H^{1/2}(N_2,E_2)$
be $D_i$-elliptic boundary conditions which correspond to each other
under the identifications given by $f$ and $\mathcal I_E$ as above.
Assume that $D_1$ and $D_2$ and their formal adjoints are coercive at infinity.

Then $D_1$, $D_2$, $D'_{1,B_1}$, and $D'_{2,B_2}$
are Fredholm operators such that
\[
  \ind D_{1} - \ind D_{2}
  = \ind D'_{1,B_1} - \ind D'_{2,B_2}
  = \int_{K_1} \alpha_{D_1} - \int_{K_2} \alpha_{D_2} ,
\]
where $\alpha_{D_1}$ and $\alpha_{D_2}$ are the index densities
associated with $D_1$ and $D_2$.
\hfill$\blacksquare$
\end{thm}

\begin{remark}
In \tref{firelind}, it is also possible to deal with the situation that $M_1$ and $M_2$
have compact boundary and elliptic boundary conditions $B_1$ and $B_2$
along their boundaries are given.
One then chooses the hypersurface $N=N_i$ such that it does not intersect
the boundary of $M_i$ and such that the boundary of $M_i$ is contained in $M_i'$.
The same arguments as above yield
$$
\ind D_{1,B_1} - \ind D_{2,B_2}
= 
\ind D'_{1,B_1\oplus B_1'} - \ind D'_{2,B_2\oplus B_2'} ,
$$
where $B_1'$ and $B_2'$ are elliptic boundary conditions along $N_1$ and $N_2$
which correspond to each other under the identifications
given by $f$ and $\mathcal I_E$ as further up. 
A similar remark applies to \tref{split}.
\end{remark}

\subsection{Boundary chiralities and index}
\label{coth}

\begin{lemma}\label{freedex}
Assume the \setup{} and that $M$ is connected.
Let $D$ be formally self-adjoint and let $A$ anticommute with $\sigma_D(\nu^\flat)$.
Let $\chi$ be a boundary chirality as in \eref{chirco1}
which commutes with $\sigma_D(\nu^\flat)$.
Let $E=E^+\oplus E^-$ be the orthogonal splitting
into the eigenbundles of $\chi$ for the eigenvalues $\pm1$,
and write
\begin{equation*}
  A = \begin{pmatrix} 0 & A^- \cr A^+ & 0 \end{pmatrix}
\end{equation*}
with respect to this splitting.
Then, if $D$ is coercive at infinity,
\begin{equation*}
  \ind D_{B_\chi} = \tfrac12\ind A^+ = -\tfrac12\ind A^- ,
\end{equation*}
where $B_\chi=H^{1/2}(\dM,E^+)$ is as in \eref{chirco1}.
\end{lemma}

\begin{proof}
Let $B_\pm=\ker A\oplus\{\phi\pm\chi\phi \mid \phi\in H^{1/2}_{(-\infty)}(A)$.
Then $B_\pm$ is a $D$-elliptic boundary condition and, by \tref{thmad},
\begin{equation*}
  \ind D_{B_\pm} =  \ind D_{\BAPS} + \dim\ker A .
\end{equation*}
By \cref{cor:AgraDy}, we have
\begin{equation*}
  \ind D_{B_{\pm\chi}} = \ind D_{B_\pm} - \dim \ker A^\mp ,
\end{equation*}
where $B_{+\chi}=B_\chi$ and $B_{-\chi}=H^{1/2}(\dM,E^-)$.
Since $B_{-\chi}=B_\chi^\perp$ and $B_{-\chi}$ is invariant under $\sigma_D(\nu^\flat)$,
we get that $B_{-\chi}$ is the adjoint of $B_\chi$.
In conclusion
\begin{align*}
  2\ind D_{B_{\chi}}
  &= \ind D_{B_{\chi}} - \ind D_{B_{-\chi}} \\
  &= \ind D_{B_{+}} - \dim\ker A^- - \ind D_{B_{-}} + \dim\ker A^+ \\
  &= \ind A^+ .
  \qedhere
\end{align*}
\end{proof}

\begin{thm}[{Cobordism Theorem, \cite[Thm.~1.22]{BB}}]\label{cobothm}
Assume the \setup{} and that $M$ is connected.
Let $D$ be formally self-adjoint and let $A$ anticommute with $\sigma_D(\nu^\flat)$.
Then $\chi=i\sigma_D(\nu^\flat)$ is a boundary chirality.
Moreover, if $D$ is coercive at infinity and with $A^\pm$ as in \tref{freedex}, then
\begin{equation*}
\ind A^+ = \ind A^- = 0 .
\end{equation*}
\end{thm}

Originally, the cobordism theorem was formulated for compact manifolds $M$
with boundary and showed the cobordism invariance of the index. 
This played an important role in the original proof of the Atiyah-Singer index theorem,
compare e.g.\ \cite[Ch.~XVII]{Pa} and \cite[Ch.~21]{BW}.
In this case,
one can also derive the cobordism invariance from Roe's index theorem
for partitioned manifolds \cite{R,Hi}.
We replace compactness of the bordism by the weaker assumption of coercivity of $D$.

\begin{proof}[Sketch of proof of \tref{cobothm}]
We show that $\ker D_{B_\chi} = \ker D_{B_{-\chi}} = 0$,
then the assertion follows from \lref{freedex}.
Let $\Phi\in\ker D_{B^\pm,\max}$.
By \tref{maxdom}.\ref{parintmd}, we have
\begin{align*}
  0
  &= (D_{\max}\Phi,\Phi)_{L^2(M)} - (\Phi,D_{\max}\Phi)_{L^2(M)} \\
  &= - (\sigma_D(\nu^\flat)\mathcal R\Phi,\mathcal R\Phi)_{L^2(\dM)} \\
  &= \pm i \|\mathcal R\Phi \|_{L^2(\dM)}^2 ,
\end{align*}
and hence $\mathcal R\Phi=0$.
Now an elementary argument involving the unique continuation
for solutions of $D\Phi=0$ implies $\Phi=0$.
\end{proof}

As an application of \lref{freedex} and \tref{cobothm},
we generalize Freed's Theorem~B from \cite{Fr} as follows:

\begin{thm}\label{freedb}
Assume the \setup{} and that $M$ is connected.
Let $D$ be formally self-adjoint and let $A$ anticommute with $\sigma_D(\nu^\flat)$.
Let $\chi$ be a boundary chirality as in \eref{chirco1}
which commutes with $\sigma_D(\nu^\flat)$.
Let $E=E^{++}\oplus E^{+-}\oplus E^{-+}\oplus E^{--}$ be the orthogonal splitting
into the simultaneous eigenbundles of $i\sigma_D(\nu^\flat)$ and $\chi$
for the eigenvalues $\pm1$.

Then $A$ maps $E^{++}$ to $E^{--}$ and $E^{+-}$ to $E^{-+}$ and conversely.
Moreover, with the corresponding notation for the restrictions of $A$, we have,
if $D$ is coercive at infinity,
\begin{equation*}
  \ind D_{B_\chi} = \ind A^{++} = - \ind A^{--} .
\end{equation*}
\end{thm}

\begin{proof}
By \tref{cobothm}, we have
\begin{equation*}
  \ind A^{++} + \ind A^{+-}  =\ind A^{--} + \ind A^{-+} = 0 .
\end{equation*}
On the other hand, $A^{--}$ is adjoint to $A^{++}$, hence \lref{freedex} gives
\begin{equation*}
  2 \ind D_{B_\chi} = \ind A^{++} + \ind A^{-+}  = \ind A^{++} - \ind A^{--} = 2 \ind A^{++} .
  \qedhere
\end{equation*}
\end{proof}

\appendix

\section{Dirac operators in the sense of Gromov and Lawson}
\label{sec:GromovLawson}

Here we discuss an important subclass of Dirac-type operators.
Note that the connection in \cref{cor:guterZshg} is not metric, in general.

\begin{definition}\label{def:GLDirac}
A formally self-adjoint operator $D:C^\infty(M,E) \to C^\infty(M,E)$ of Dirac type is called a \emph{Dirac operator in the sense of Gromov and Lawson} if there exists a metric connection $\nabla$ on $E$ such that 
\begin{asparaenum}[1)]
\item 
$D = \sum\nolimits_j \sigma_D(e_j^*)\circ\nabla_{e_j}$,
for any local orthonormal tangent frame $(e_1,\ldots,e_n)$;
\item
the principal symbol $\sigma_D$ of $D$ is parallel with respect to $\nabla$ and to the Levi-Civita connection.
\end{asparaenum}
\end{definition}

This is equivalent to the definition of \emph{generalized Dirac operators} in \cite[Sec.~1]{GL} or to \emph{Dirac operators on Dirac bundles} in \cite[Ch.~II, \S~5]{LM}.

\begin{remark}
For a Dirac operator in the sense of Gromov and Lawson, the connection $\nabla$ in \dref{def:GLDirac} and the connection in the Weitzenb\"ock formula \eqref{eq:Weitzenboeck} coincide and are uniquely determined by these properties.
We will call $\nabla$ the connection \emph{associated with the Dirac operator} $D$.
Moreover, the endomorphism field $\mathcal{K}$ in the Weitzenb\"ock formula takes the form

\begin{equation*}
\mathcal K 
= 
\frac12\sum_{i,j} \sigma_D(e_i^*)\circ \sigma_D(e_j^*)\circ R^\nabla(e_i,e_j)
\end{equation*}
where $R^\nabla$ is the curvature tensor of $\nabla$.
See \cite[Prop.~2.5]{GL} for a proof.
\end{remark}

Next, we show how to explicitly construct an adapted operator on the boundary satisfying \eqref{eq:ANuAnti} for a Dirac operator in the sense of Gromov and Lawson.
Let $\nabla$ be the associated connection.
Along the boundary we define
\begin{equation}
A_0 
:= 
\sigma_D(\nu^\flat)^{-1}D - \nabla_\nu = \sigma_D(\nu^\flat)^{-1}\sum_{j=2}^n \sigma_D(e_j^*)\nabla_{e_j} .
\label{eq:defA0}
\end{equation}
Here $(e_2,\ldots,e_n)$ is any local tangent frame for $\dM$.
Then $A_0$ is a first-order differential operator acting on section of $E|_\dM \to \dM$ with principal symbol $\sigma_{A_0}(\xi)=\sigma_D(\nu^\flat)^{-1}\sigma_D(\xi)$ as required for an adapted boundary operator.
From the Weitzenb\"ock formula \eqref{eq:Weitzenboeck} we get, using \pref{greenfor} twice, once for $D$ and once for $\nabla$, for all $\Phi,\Psi \in C^\infty_c(M,E)$:
\begin{align}
0 &=
\int_M \big( \langle D^2\Phi,\Psi\rangle - \langle \nabla^*\nabla \Phi,\Psi\rangle - \langle \mathcal{K}\Phi,\Psi\rangle\big)\dV \nonumber \\
&=
\int_M \big( \langle D\Phi,D\Psi\rangle - \langle \nabla \Phi,\nabla\Psi\rangle - \langle \mathcal{K}\Phi,\Psi\rangle\big)\dV \nonumber  \\
&\quad+\int_\dM \big( -\langle \sigma_D(\nu^\flat)D\Phi,\Psi\rangle + \langle \sigma_{\nabla^*}(\nu^\flat)\nabla \Phi,\Psi\rangle \big)\dS .
\label{eq:pf1}
\end{align}
For the boundary contribution we have
\begin{align}
 -\langle \sigma_D(\nu^\flat)D\Phi,\Psi\rangle + \langle \sigma_{\nabla^*}(\nu^\flat)\nabla \Phi,\Psi\rangle 
&=
\langle \sigma_D(\nu^\flat)^{-1}D\Phi,\Psi\rangle - \langle \nabla \Phi,\sigma_{\nabla}(\nu^\flat)\Psi\rangle \nonumber \\
&=
\langle \sigma_D(\nu^\flat)^{-1}D\Phi,\Psi\rangle - \langle \nabla \Phi,\nu^\flat\otimes\Psi\rangle\nonumber  \\
&=
\langle \sigma_D(\nu^\flat)^{-1}D\Phi,\Psi\rangle - \langle \nabla_\nu \Phi,\Psi\rangle\nonumber  \\
&=
\langle A_0\Phi,\Psi\rangle .
\label{eq:pf2}
\end{align}
Inserting \eqref{eq:pf2} into \eqref{eq:pf1} we get
\begin{equation}
\int_M \big( \langle D\Phi,D\Psi\rangle - \langle \nabla \Phi,\nabla\Psi\rangle - \langle \mathcal{K}\Phi,\Psi\rangle\big)\dV
=
-\int_\dM\langle A_0\phi,\psi\rangle\dS
\label{eq:pf3}
\end{equation}
where $\phi:=\Phi|_\dM$ and $\psi:=\Psi|_\dM$.
Since the left-hand side of \eqref{eq:pf3} is symmetric in $\Phi$ and $\Psi$, the right-hand side is symmetric as well, hence $A_0$ is formally self-adjoint.
This shows that $A_0$ is an adapted boundary operator for $D$.

In general, $A_0$ does not anticommute with $\sigma_D(\nu^\flat)$ however.
We will rectify this by adding a suitable zero-order term.
First, let us compute the anticommutator of $A_0$ and $\sigma_D(\nu^\flat)$:
\begin{align*}
\{\sigma_D(\nu^\flat),A_0\}\phi
&=
\sum_{j=2}^n \sigma_D(e_j^*)\nabla_{e_j} \phi + 
 \sigma_D(\nu^\flat)^{-1}\sum_{j=2}^n \sigma_D(e_j^*)\nabla_{e_j}(\sigma_D(\nu^\flat)\phi) \\
&=
\sum_{j=2}^n\Big( \sigma_D(e_j^*)\nabla_{e_j}\phi +  \sigma_D(\nu^\flat)^{-1} \sigma_D(e_j^*)\sigma_D(\nu^\flat)\nabla_{e_j}\phi \\
&\quad\quad\quad
+  \sigma_D(\nu^\flat)^{-1} \sigma_D(e_j^*)\sigma_D(\nabla_{e_j}\nu^\flat)\phi   \Big) \\
&=
\sigma_D(\nu^\flat)^{-1}\sum_{j=2}^n \sigma_D(e_j^*)\sigma_D(\nabla_{e_j}\nu^\flat)\phi .
\end{align*}
Now $\nabla_\cdot\nu$ is the negative of the Weingarten map of the boundary with respect to the normal field $\nu$.
We choose the orthonormal tangent frame $(e_2,\ldots,e_n)$ to consist of eigenvectors
of the Weingarten map.
The corresponding eigenvalues $\kappa_2,\ldots,\kappa_n$ are the \emph{principal curvatures} of $\dM$.
We get
\begin{align*}
\sum_{j=2}^n\sigma_D(e_j^\flat)\sigma_D(\nabla_{e_j}\nu^\flat)
=
-\sum_{j=2}^n \sigma_D(e_j^\flat)\sigma_D(\kappa_j e_j^\flat) 
=
\sum_{j=2}^n \kappa_j 
=
(n-1)H,
\end{align*}
where $H$ is the \emph{mean curvature} of $\dM$ with respect to $\nu$.
Therefore,
\begin{equation*}
\{\sigma_D(\nu^\flat),A_0\} =(n-1)H \sigma_D(\nu^\flat)^{-1} = -(n-1)H \sigma_D(\nu^\flat) .
\end{equation*}
Since clearly
\[
\{\sigma_D(\nu^\flat),(n-1)H\} = 2(n-1)H \sigma_D(\nu^\flat),
\]
the operator
\begin{equation*}
  A:=A_0 + \frac{n-1}{2}H=\sigma_D(\nu^\flat)^{-1}D - \nabla_\nu + \frac{n-1}{2}H 
\end{equation*}
is an adapted boundary operator for $D$
satisfying \eqref{eq:ANuAnti}.
From \eqref{eq:pf3} we also have
\begin{equation}
\int_M \big( \langle D\Phi,D\Psi\rangle - \langle \nabla \Phi,\nabla\Psi\rangle - \langle \mathcal{K}\Phi,\Psi\rangle\big)\dV
=
\int_\dM\langle (\tfrac{n-1}{2}H-A)\phi,\psi\rangle\dS .
\label{eq:pf4}
\end{equation}

\begin{definition}
For a Dirac operator $D$ in the sense of Gromov and Lawson as above,
we call $A$ the \emph{canonical boundary operator} for $D$.
\end{definition}

\begin{remark}
The canonical boundary operator $A$ is again a Dirac operator in the sense of Gromov and Lawson.
Namely, define a connection on $E|_{\dM}$ by
\[
\nd_X\phi := \n_X\phi + \textstyle{\frac12}\sigma_D(\nu^\flat)^{-1}\sigma_D(\nabla_X\nu^\flat)\phi .
\]
The Clifford relations \eqref{Cliff1} show that the term $\sigma_D(\nu^\flat)^{-1}\sigma_D(\nabla_X\nu^\flat)=\sigma_D(\nu^\flat)^*\sigma_D(\nabla_X\nu^\flat)$ is skewhermitian, hence $\nd$ is a metric connection.
By \eqref{eq:defA0}, $A_0=\sum_{j=2}^n \sigma_{A_0}(e_j^*)\circ\n_{e_j}$.
This, $\sigma_{A_0}=\sigma_{A}$, and 
\[
\sum_{j=2}^n \sigma_{A_0}(e_j^*)\sigma_D(\nu^\flat)^*\sigma_D(\nabla_{e_j}\nu^\flat)
=
\frac{n-1}{2}\, H
\]
show that 
\[
A=\sum_{j=2}^n \sigma_{A}(e_j^*)\circ\nd_{e_j} .
\]
Moreover, a straightforward computation using the Gauss equation for the Levi-Civita connections $\n_X\xi=\nd_X\xi -\xi(\n_X\nu)\nu^\flat$ shows that $\sigma_A$ is parallel with respect to the boundary connections $\nd$.
\end{remark}

\begin{remark}
The triangle inequality and the Cauchy--Schwarz inequality show
\begin{align}
|D\Phi |^2 
&=
\big|\sum_{j=1}^n\sigma_D(e_j^\flat)\n_{e_j}\Phi\big|^2 
\le 
\big(\sum_{j=1}^n|\sigma_D(e_j^\flat)\n_{e_j}\Phi|\big)^2 \notag\\
&\le
n\cdot\sum_{j=1}^n|\sigma_D(e_j^\flat)\n_{e_j}\Phi|^2 
=
n\cdot\sum_{j=1}^n\langle\sigma_D(e_j^\flat)^*\sigma_D(e_j^\flat)\n_{e_j}\Phi,\n_{e_j}\Phi\rangle \notag\\
&=
n\cdot \sum_{j=1}^n|\n_{e_j}\Phi|^2 
= 
n\cdot | \nabla\Phi |^2 ,
\label{cauchydna}
\end{align}
for any orthonormal tangent frame $(e_1,\dots,e_n)$ and all $\Phi\in C^\infty(M,E)$.

When does equality hold?
Equality in the Cauchy--Schwarz inequality implies that all summands $|\sigma_D(e_j^\flat)\n_{e_j}\Phi|$ are equal, i.e., $|\sigma_D(e_j^\flat)\n_{e_j}\Phi|=|\sigma_D(e_1^\flat)\n_{e_1}\Phi|$.
Equality in the triangle inequality then implies $\sigma_D(e_j^\flat)\n_{e_j}\Phi=\sigma_D(e_1^\flat)\n_{e_1}\Phi$ for all $j$.
Thus
\[
\sigma_D(e_1^\flat)\n_{e_1}\Phi 
=
\frac1n \sum_{j=1}^n  \sigma_D(e_j^\flat)\n_{e_j}\Phi 
=
\frac1n D\Phi,
\]
hence $\n_{e_1}\Phi = \frac1n\sigma_D(e_1^\flat)^* D\Phi$.
Since $e_1$ is arbitrary, this shows the \emph{twistor equation} 
\begin{equation}
  \nabla_X\Phi = \textstyle{\frac1{n}} \sigma_D(X^\flat)^* D\Phi ,
  \label{twist}
\end{equation}
for all vector fields $X$ on $M$.
Conversely, if $\Phi$ solves the twistor equation, one sees directly that equality holds in \eqref{cauchydna}.

Inserting \eqref{cauchydna} into \eqref{eq:pf4} yields
\begin{equation*}
\tfrac{n-1}{n} \int_M | D\Phi |^2 \dV
\ge 
\int_M \langle \mathcal K\Phi,\Phi \rangle\dV + \int_\dM \langle (\tfrac{n-1}{2}H - A)\phi,\phi\rangle\dS ,
\end{equation*}
for all $\Phi\in\Cu_c(M,E)$,
where $\phi:=\Phi|_{\dM}$.
Moreover, equality holds if and only if $\Phi$ solves the twistor equation \eqref{twist}.
\end{remark}

\section{Proofs of some auxiliary results}
\label{sec:beweise}

In this section we collect the proofs of some of the auxiliary results.

\begin{proof}[Proof of Proposition~\ref{prop:guterZshg}]
We start by choosing an arbitrary connection $\bar\nabla$ on $E$ and define
\begin{equation*}
  \bar D:  C^\infty(M,E) \to C^\infty(M,F) , \quad
  \bar D\Phi := \sum\nolimits_j \sigma_D(e_j^*)\bar\nabla_{e_j}\Phi .
\end{equation*}
Then $\bar D$ has the same principal symbol as $D$ and, therefore,
the difference $S:= D - \bar D$ is of order $0$.
In other words, $S$ is a field of homomorphisms from $E$ to $F$.

Since $\mathcal{A}_D$ is onto, the restriction $\mathcal{A}$ of $\mathcal{A}_D$ to the orthogonal complement of the kernel of $\mathcal{A}_D$ is a fiberwise isomorphism.
We put $V:=\mathcal{A}^{-1}(S)$ and define a new connection by
\[
 \nabla := \bar\nabla + V .
\]
We compute
\begin{align*}
  \sum\nolimits_j \sigma_D(e_j^*)\circ \nabla_{e_j}
  &= 
  \sum\nolimits_j \sigma_D(e_j^*)\circ\bar\nabla_{e_j} 
  + \sum\nolimits_j \sigma_D(e_j^*)\circ V(e_j) \\
  &=
  \bar D + \mathcal{A}_D(V)\\
  &=
  \bar D + S 
  =
  D . \qedhere
\end{align*}
\end{proof}

\begin{proof}[Proof of \pref{Weitzen}]
Let $\widetilde{\nabla}$ be any metric connection on $E$. 
Then $F:=D^*D-\widetilde{\nabla}^*\widetilde{\nabla}$ is formally self-adjoint.
Since both, $D^*D$ and $\widetilde{\nabla}^* \widetilde{\nabla}$, have the same principal symbol $-|\xi|^2\cdot\id$, the operator $F$ is of order at most one. 
Any other metric connection $\nabla$ on $E$ is of the form $\nabla = \widetilde{\nabla} + B$ where $B$ is a $1$-form with values in skewhermitian endomorphisms of $E$.
Hence
   \begin{align*}
      D^*D &= (\nabla-B)^*(\nabla-B) + F 
         = \nabla^* \nabla \underbrace{- \nabla^* B - B^* \nabla + B^* B + F}_{=:\,\mathcal{K}} .
   \end{align*}
In general, $\mathcal{K}$ is of first order and we need to show that there is a unique $B$ such that $\mathcal{K}$ is of order zero.
Since $B^* B$ is of order zero, $\mathcal{K}$ is of order zero if and only if $F - \nabla^*B - B^*\nabla$ is of order zero, i.e., if and only if $\sigma_F(\xi) = \sigma_{\nabla^*B + B^*\nabla}(\xi)$ for all $\xi \in T^*M$.
We compute, using a local tangent frame $e_1,\ldots,e_n$,
   \begin{align*}
      \<\sigma_{\nabla^* B + B^* \nabla}(\xi)\varphi,\psi\>
         &= \<\big(\sigma_{\nabla^*}(\xi)\circ B + B^*\circ \sigma_{\nabla}(\xi)\big)\varphi,\psi\> \\
         &= - \<B\varphi,\sigma_\nabla(\xi)\psi\> + \<\sigma_\nabla(\xi)\varphi,B\psi\> \\
         &= - \<B\varphi,\xi\otimes \psi\> + \<\xi\otimes \varphi,B\psi\> \\
         &= -\big\langle\sum_i e_i^* \otimes B_{e_i}\varphi,\xi \otimes \psi\big\rangle + \big\langle\xi\otimes \varphi,\sum_i e_i^* \otimes B_{e_i} \psi\big\rangle  \\
         &= - \sum_i \langle e_i^*,\xi\rangle\langle B_{e_i}\varphi,\psi\rangle + \sum_i \langle e_i^*,\xi\rangle \<\varphi,B_{e_i} \psi\> \\
         &= - \< B_{\xi^{\sharp}}\varphi,\psi\> + \< \varphi,B_{\xi^{\sharp}} \psi\> \\
         &= - 2 \<B_{\xi^{\sharp}}\varphi,\psi\>.
   \end{align*}
Hence, $\sigma_{\nabla^* B + B^* \nabla}(\xi) = - 2 B_{\xi^{\sharp}}$. 
Thus, $\mathcal{K}$ is of order 0 if and only if 
\[
B_X = - \tfrac{1}{2}\, \sigma_F(X^b)
\] 
for all $X\in TM$.
Note that $\sigma_F(\xi)$ is indeed skewhermitian
because $F$ is formally self-adjoint.
\end{proof}
\begin{proof}[Proof of \lref{norma}]
Since $D$ is formally self-adjoint and of Dirac type,
\begin{equation}
  - \sigma_D(\nu^\flat) = \sigma_D(\nu^\flat)^* = \sigma_D(\nu^\flat)^{-1} ,
\label{symbnu}
\end{equation}
by \eqref{symbad} and \eqref{symbdt}.
Let $A_0$ be adapted to $D$ along $\partial M$ and $\xi\in T^*_x\partial M$,
as usual extended to $T^*_xM$ by $\xi(\nu(x))=0$.
Then, again using \eqref{Cliff1} and \eqref{symba},
\begin{align*}
  \sigma_{A_0}(\xi) + \sigma_D(\nu(x)^\flat)&\sigma_{A_0}(\xi)\sigma_D(\nu(x)^\flat)^* \\
  &= \sigma_D(\nu(x)^\flat)^{-1}\sigma_D(\xi) + \sigma_D(\xi)\sigma_D(\nu(x)^\flat)^* \\
  &= \sigma_D(\nu(x)^\flat)^*\sigma_D(\xi) + \sigma_D(\xi)^*\sigma_D(\nu(x)^\flat) \\
  &= 2\langle \nu(x)^\flat , \xi \rangle \cdot \id_E \\
  &= 0 .
\end{align*}
Hence $2S:=A_0+\sigma_D(\nu^\flat)A_0\sigma_D(\nu^\flat)^*$ is of order $0$,
that is, $S$ is a field of endomorphisms of $E$ along $\partial M$.
Since $A_0$ is formally self-adjoint so is $S$ and, by \eqref{symbnu},
\begin{align*}
  \sigma_D(\nu^\flat) 2S
  = \sigma_D(\nu^\flat) A_0 + A_0 \sigma_D(\nu^\flat)
  = 2S \sigma_D(\nu^\flat) .
\end{align*}
Hence $A:=A_0-S$ is adapted to $D$ along $\partial M$ and
\begin{align*}
\sigma_D(\nu^\flat) A + A \sigma_D(\nu^\flat)
&=
\sigma_D(\nu^\flat) A_0 + A_0 \sigma_D(\nu^\flat) - \sigma_D(\nu^\flat) S - S \sigma_D(\nu^\flat)  \\
&=
\sigma_D(\nu^\flat)\big(A_0 -\sigma_D(\nu^\flat)A_0\sigma_D(\nu^\flat) -2S\big) \\
&=
\sigma_D(\nu^\flat)\big(A_0 +\sigma_D(\nu^\flat)A_0\sigma_D(\nu^\flat)^* -2S\big) \\
&= 0.
\qedhere
\end{align*}
\end{proof}



\end{document}